\documentclass[a4paper]{article}

\usepackage[backend=bibtex, maxnames=4, style=numeric, isbn=false]{biblatex}
\usepackage{a4wide,aliascnt,amsmath,amssymb,amsthm,bbm,bm,booktabs,enumitem,hyperref}

\numberwithin{counter}{section}

\theoremstyle{plain}

\newaliascnt{theorem}{counter}
\newtheorem{theorem}[theorem]{Theorem}
\aliascntresetthe{theorem}

\newaliascnt{proposition}{counter}
\newtheorem{proposition}[proposition]{Proposition}
\aliascntresetthe{proposition}

\newaliascnt{result}{counter}

\aliascntresetthe{result}

\newaliascnt{lemma}{counter}
\newtheorem{lemma}[lemma]{Lemma}
\aliascntresetthe{lemma}

\newaliascnt{corollary}{counter}

\aliascntresetthe{corollary}

\newaliascnt{exercise}{counter}

\aliascntresetthe{exercise}

\newaliascnt{classexercise}{counter}

\aliascntresetthe{classexercise}

\theoremstyle{definition}

\newaliascnt{definition}{counter}

\aliascntresetthe{definition}

\newaliascnt{remark}{counter}
\newtheorem{remark}[remark]{Remark}
\aliascntresetthe{remark}

\newaliascnt{example}{counter}
\newtheorem{example}[example]{Example}
\aliascntresetthe{example}

\renewcommand\*{\cdot}

\newcommand\R{\mathbb{R}}
\newcommand\C{\mathbb{C}}
\newcommand\N{\mathbb{N}}

\newcommand{\trace}{\operatorname{tr}}     
\newcommand\Id{\operatorname{Id}}                               
\newcommand\dom{{D}}                                      

\providecommand{\norm}[2]{\left\lVert#1\right\rVert_{#2}}

\providecommand{\abs}[1]{\left\lvert#1\right\rvert}
\newcommand{\smashedtilde}[1]{\vphantom{#1}\smash{\tilde{#1}}}
\newcommand{\mednorm}[2]{\big\lVert#1\big\rVert_{#2}}
\newcommand{\bignorm}[2]{\bigg\lVert#1\bigg\rVert_{#2}}
\newcommand{\ip}[3]{\left\langle #1, #2\right\rangle_{#3}}

\renewcommand{\Re}{\operatorname{Re}}
\renewcommand{\Im}{\operatorname{Im}}

\newcommand{\borel}[1]{\mathcal B\left(#1\right)}
\newcommand{\Lip}{\operatorname{Lip}}
\newcommand{\RR}{\mathbb{R}}

\newcommand{\PP}{\mathbb{P}}

\newcommand{\NN}{\mathbb{N}}

\newcommand{\EE}{\mathbb{E}}

\newcommand{\UU}{\mathbb{U}}

\newcommand{\cB}{\mathcal{B}}

\newcommand{\cL}{\mathcal{L}}

\newcommand\ddt{\tfrac{\partial}{\partial t}}

\newcommand\ddx{\tfrac{\partial}{\partial x}}
\newcommand\ddxx{\tfrac{\partial^2}{\partial x^2}}

\newcommand\F{\mathcal{F}}

\renewcommand\d{\,\operatorname{d}\hspace{-0.05cm}}

\renewcommand{\P}[1]{\mathbb{P}\left[#1\right]}

\newcommand{\E}[1]{\mathbb{E}\left[#1\right]}

\newcommand\HS{L_2}

\begin{document}
\title{Weak convergence rates for stochastic evolution equations and applications to nonlinear stochastic wave, HJMM, stochastic Schr\"odinger and linearized stochastic Korteweg--de Vries equations\thanks{We are grateful to Arnulf Jentzen for pointing us to the topic and sharing with us ideas from unpublished joint work with Sonja Cox \cite{cox2017weak}. M.\,S.\,M.~thankfully acknowledges support by the Swiss National Science Foundation through grant SNF $205121\_163425$. P.\,H.~thankfully acknowledges support by the Freiburg Institute of Advanced Studies in the form of a Junior Fellowship.}}

\author{Philipp Harms \and Marvin S.~M\"uller}

\AtEndDocument{
  \bigskip{
    \footnotesize
    
    \textsc{P.~Harms, Freiburg Institute of Advanced Studies and Department of Mathematical Stochastics, University of Freiburg, Germany}\par\nopagebreak
    \textit{E-mail address:} \texttt{philipp.harms@stochastik.uni-freiburg.de}
    
    \medskip
    
    \textsc{M.~S.~M\"uller, Department of Mathematics, ETH Zurich,
      Switzerland}\par\nopagebreak
    \textit{E-mail address:} \texttt{marvin.mueller@math.ethz.ch}
  }
}

\maketitle

\begin{abstract}
  We establish weak convergence rates for noise discretizations of a wide class of stochastic evolution equations with non-regularizing semigroups and additive or multiplicative noise. 
  This class covers the nonlinear stochastic wave, HJMM, stochastic Schr\"odinger and linearized stochastic Korteweg--de Vries equation. 
  For several important equations, including the stochastic wave equation, previous methods give only suboptimal rates, whereas our rates are essentially sharp.
\end{abstract}

\tableofcontents

\section{Introduction}
\label{sec:introduction}

This paper establishes weak convergence rates for noise discretizations of a wide class of stochastic evolution equations with non-regularizing semigroups and regular non-linearities. 
We confirm that the weak convergence rate equals twice the strong convergence rate and is characterized in terms of two components: 
\begin{enumerate}[noitemsep,nolistsep,label=\alph*)]
\item the decay of the covariance of the noise, and
\item regularity of the solution, encoded in the choice of an invariant subspace.
\end{enumerate} 
In the case of additive noise the upper bound on the weak error is sharp. 

Our result is motivated by the study of stochastic partial differential equations driven by infinite dimensional noise processes, which came up in a large variety of applications. Numerical simulations of these equations require a full discretization in space, time, and noise. Our result complements the results on spatial and temporal discretizations in \cite{cox2017weak,jacobe2015weak,jacobe2018lower} and thereby completes the weak error analysis of numerical discretizations of the above-mentioned class of equations, providing a full picture of their complexity (see~\autoref{tab:rates}). Weak (as opposed to strong) convergence rates offer a flexible way of measuring the quality of the approximation, as the class of test functions can be chosen to reflect the priorities at application level. 

Our results are general and apply to a variety of equations, as we demonstrate in several examples. 
For the non-linear stochastic wave equation with additive or multiplicative space-time white noise they give the essentially sharp rate $1-\epsilon$, $\epsilon >0$. 
Further examples are the Heath--Jarrow--Morton--Musiela (HJMM) equation, the stochastic Schr\"odinger equation, and the linearized stochastic Korteweg--de Vries equation.

The proof of our main result works as follows. 
First, we regularize the equation using Yosida approximations of the semigroup, as this allows us to work with the Kolmogorov equation and the strong It\^o formula.
Second, we express the weak discretization error of the regularized equation in terms of the solution of the Kolmogorov equation. 
Third, we introduce an additional subspace, which links the regularity of the solution to the quality of the noise approximation and determines the error rate. 
This last step is essential for obtaining optimal rates in many examples, including the stochastic wave equation, and constitutes an important theoretical contribution of this paper. 

In the previous literature weak convergence rates have been studied intensively for equations of the above type with regularizing semigroups; see e.g.\@ \cite{kruse2014strong} and references therein. 
However, in the case of non-regularizing semigroups there remain many open questions. 
While temporal and spatial discretizations have been studied in \cite{hausenblas2010weak, kovacs2012weak, kovacs2013weak, kovacs2015weak}
for additive noise and in
\cite{cox2017weak,jacobe2015weak,jacobe2018lower,debouard2006weak, doersek2010semigroup, doersek2013efficient, hefter2016weak, krivko2013numerical, wang2015exponential} 
for multiplicative noise, this is the first result on the discretization of multiplicative noise in this setting. 
Moreover, our framework is general and encompasses a variety of equations from mathematical finance and physics. 

\begin{table}[h]
\centering
\begin{tabular}{cccc}\toprule
& \multicolumn{3}{c}{discretization}
\\\cmidrule(r){2-4}
noise coefficient  & time & space & noise 
\\\midrule
constant & 1 & 1 & 1
\\
affine  & 1 & 1 & \phantom{${}^*$}1${}^*$
\\
non-affine & (?) & (?)& (?)
\\\bottomrule
\end{tabular}
\caption{Weak convergence rates of exponential Euler and Galerkin discretizations of the stochastic wave equation with space time white noise. The rates are to be understood as $1-\epsilon$ for arbitrary $\epsilon>0$. References are given in \autoref{sec:introduction}. The starred rate is a result of this paper (see \autoref{prop:wave}).}
\label{tab:rates}
\end{table}

\subsection{Notation}
\label{sec:notation}

$\N$ denotes the natural numbers without zero.
Function spaces are denoted as follows: 
$B$ denotes bounded functions with the supremum norm, 
$C$ denotes continuous functions,
$C_b$ denotes continuous bounded functions with the supremum norm,
$C^k$ denotes continuous functions which are $k$-times Fr\'echet differentiable on the interior of the domain and whose derivatives up to order $k$ extend to continuous functions on the domain,
$C^k_b$ denotes the subset of $C^k$ whose derivatives of orders $1$ to $k$ belong to $C_b$ with norm $\|f\|_{C^k_b}=\|f(0)\|+\|f'\|_{C_b}+\dots+\|f^{(k)}\|_{C_b}$, 
$L$ denotes linear operators with the operator norm,
$\Lip$ denotes Lipschitz functions with norm $\norm{f}{\Lip} = \|f(0)\|+\sup_{x\neq y}\|f(x)-f(y)\|/\|x-y\|$, 
$\cL^p$ denotes strongly measurable $p$-integrable functions,
$L^p$ denotes the corresponding equivalence classes modulo equality almost surely,
$L^{(2)}$ denotes bilinear operators,
$L_2$ denotes Hilbert--Schmidt operators,
$W^{\alpha,p}$ denotes the Sobolev--Slobodeckij space with smoothness parameter $\alpha$ and integrability parameter $p$, 
$H^\alpha=W^{\alpha,2}$ denotes the Bessel potential space, 
and $H^\alpha_0$ denotes the closure of the compactly supported smooth functions in $H^\alpha$.
For any Hilbert space $H$,
$\cB(H)$ denotes the Borel $\sigma$-algebra on $H$, 
$[X]_{\PP,\borel{H}}\in L^0(\Omega;H)$ denotes the $\PP$-equivalence class of $X\in \cL^0(\Omega;H)$,
and $\sigma_P(A)$ denotes the point spectrum of a linear operator $A\colon D(A) \subseteq H \to H$.
Note that we do not require functions in $C^k_b$ to be bounded. 

\subsection{Main result}\label{sec:result}

The following theorem establishes weak and strong convergence rates for noise discretizations of a certain class of stochastic evolution equations. 
Roughly speaking, the assumptions of the theorem guarantee that the equation is well-posed on a Hilbert space $H$ and a continuously embedded subspace $V$ of $H$, and that Kolmogorov's backward equation for $H$-valued solutions is well-posed.
This is used to bound the weak and strong discretization errors in terms of the Hilbert--Schmidt norm of the difference between the actual and the discretized volatility.
The role of the subspace $V$ is discussed in Section~\ref{sec:regularity} below.
An extended version of the theorem with explicit bounds is presented in Proposition~\ref{prop:weak_error_strong}.
Note that the function $\phi \in C^2_b(H;\R)$ in the theorem may be unbounded according to our definition.

\begin{theorem}\label{thm:weak_errorG}
Let $T\in (0,\infty)$,  
let $(H,\ip{\cdot}{\cdot}{H})$, $(U, \ip{\cdot}{\cdot}{U})$, and $(V,\ip{\cdot}{\cdot}{V})$ be
separable $\RR$-Hilbert spaces with $V\subset H$ densely and
continuously, let $(\Omega, \mathcal F,
(\mathcal F_t)_{t \in [0,T]},\mathbb P)$ be a stochastic basis,
let $(W_t)_{t\in [0,T]}$ be an
$\Id_U$-cylindrical $(\F_t)$-Wiener process, let
$S\colon[0,\infty)\to L(H)$ be a strongly continuous semigroup, which
restricts to a strongly continuous semigroup
5
 $S|_V\colon[0,\infty)\to L(V)$,
let $F\in C^2_b(H)$ and $B\in C^2_b(H; \HS(U; H))$ be such that $F(V)\subseteq V$, $B(V)\subseteq \HS(U;V)$, and the mappings $V\ni x\mapsto F(x)\in V$ and $V\ni x \mapsto B(x) \in \HS(U;V)$ are Lipschitz continuous,    
let $\xi \in \cL^2(\Omega;V)$ be $\F_0\slash \borel{V}$-measurable,
let $(e_k)_{k\in\N}$ be an orthonormal basis of $U$, 
for each $n\in\N\cup\{\infty\}$ let $P_n \in L(U)$ be the orthogonal projection onto the closure of the linear span of $\{e_k\colon k\in \N\cap [0,n)\}$, 
and let $X^n\colon [0,T]\times \Omega \to H$ be a predictable process which satisfies that $\P{\int_0^T \norm{X^n_t}{H}^2\d t<\infty} = 1$ and for all $t \in [0,T]$,
\begin{equation*}
  [X^n_t]_{\PP,\borel{H}} = \left[S_t\xi + \int_0^t S_{t-s} F(X^n_s) \d s\right]_{\PP,\borel{H}} + \int_0^t S_{t-s} B(X^n_s) P_n \d W_s.
\end{equation*}
Then there exists $C \in (0,\infty)$ such that for each $n \in \mathbb N$ and $\phi \in C^2_b(H;\R)\setminus\{0\}$, 
\begin{equation}
\label{eq:7}
\norm{X_T^\infty- X_T^n}{L^2(\Omega;H)}^2  +
\frac{\abs{ \E{\phi\left(X^\infty_T\right)} -
\E{\phi\left(X^n_T\right)}}}{\norm{\phi}{C^2_b(H;\R)}}
\leq C \sup_{x \in V} \frac{\sum_{k=n}^\infty \norm{B(x)e_k}{H}^2 }{1+\norm{x}{V}^2}.
\end{equation}
\end{theorem}

\begin{proof}
It may be assumed without loss of generality that $S\colon[0,\infty)\to L(H)$ is uniformly bounded by adding and subtracting a multiple of the identity to the generator of the semigroup and the nonlinear part of the drift, respectively. 
Then the assumptions of Section~\ref{sec:setting} hold. 
Thus, the bound of the weak error follows from Proposition~\ref{prop:weak_error_strong}, noting that it holds for each $x \in V$ that
\begin{equation*}
\sum_{k=0}^\infty\norm{B(x)(P_\infty+P_n)e_k}{H}\norm{B(x)(P_\infty-P_n)e_k}{H}
=
2\sum_{k=n}^\infty\norm{B(x)e_k}{H}^2.
\end{equation*}
It remains to bound the strong error. For each $t \in [0,T]$, it holds that
\begin{align*}
\norm{X^\infty_t-X^n_t}{L^2(\Omega;H)}
&\leq 
\int_0^t \norm{S_{t-s}(F(X^\infty_s)-F(X^n_s))}{L^2(\Omega;H)} \d s
\\&\qquad
+\sqrt{ \int_0^t \norm{S_{t-s}(B(X^\infty_s)-B(X^n_s)P_n)}{L^2(\Omega;\HS(U;H))}^2 \d s}. 
\end{align*}
{\allowdisplaybreaks
Taking the square yields
\begin{align*}
\hspace{2em}&\hspace{-2em}
\norm{X^\infty_t-X^n_t}{L^2(\Omega;H)}^2
\leq 
\norm{S}{B([0,T];L(H))}^2 \bigg( 2T\int_0^t \norm{F(X^\infty_s)-F(X^n_s)}{L^2(\Omega;H)}^2 \d s
\\&\qquad
+4\int_0^t \norm{B(X^\infty_s)-B(X^\infty_s)P_n}{L^2(\Omega;\HS(U;H))}^2 \d s 
\\&\qquad
+4\int_0^t \norm{B(X^\infty_s)P_n-B(X^n_s)P_n}{L^2(\Omega;\HS(U;H))}^2 \d s \bigg)
\\&\leq
\norm{S}{B([0,T];L(H))}^2 \Big(2T\norm{F}{\Lip(H;H)}^2+4\norm{B}{\Lip(H;\HS(U;H))}^2\Big)\int_0^t \norm{X^\infty_s-X^n_s}{L^2(\Omega;H)}^2 \d s
\\&\qquad
+4\norm{S}{B([0,T];L(H))}^2 T\left(1+\norm{X^\infty_s}{B([0,T];L^2(\Omega;V))}^2\right)\sup_{x\in V} \frac{\sum_{k=n}^\infty \norm{B(x)e_k}{H}^2 }{1+\norm{x}{V}^2}. 
\end{align*}
}
Both sides of this inequality are finite by Lemma~\ref{lem:apriori}.\ref{item:apriori2} with $H$ replaced by $V$, and the strong rate follows from Gronwall's lemma.
\end{proof}

\subsection{Convergence rate and regularity of the solution}\label{sec:regularity}

The space $V$ in Theorem~\ref{thm:weak_errorG} can be used to encode regularity properties of the solution which go beyond those present in $H$. Choosing $V$ strictly smaller than $H$ allows one to extract a stronger convergence rate from Theorem~\ref{thm:weak_errorG} in some cases, as the following example demonstrates. 

\begin{example}
Let $H = U = L^2((0,1))$, 
for each $x\in H$ and $u\in U$ let $B(x)u$ be the constant function $B(x)u=(s\mapsto \int_0^1 x(r) u(r) \d r)$,
let $\Delta \colon H^2((0,1))\cap H^1_0((0,1)) \to L^2((0,1))$ be the Dirichlet Laplacian on $H$, 
let $\delta \in (0,1/4)$, 
let $V$ be the domain of $(-\Delta)^\delta$ with $\norm{\cdot}{V}=\norm{(-\Delta)^\delta(\cdot)}{H}$,
and for each $k \in \mathbb N$ let $e_k = (\sqrt{2}\sin( k \pi s))_{s \in (0,1)} \in H$.
Then $B \in L(H;\HS(U;H))$, $B|_V \in L(V;\HS(U;V))$, 
and it holds for each $n \in \mathbb N$ that
\begin{equation*}
\sup_{x \in H} \frac{\sum_{k=n}^\infty \norm{B(x)e_k}{H}^2 }{1+\norm{x}{H}^2 } = 1, 
\qquad
\sup_{x \in V} \frac{\sum_{k=n}^\infty \norm{B(x)e_k}{H}^2 }{1+\norm{x}{V}^2 } \leq \pi^{-4\delta} n^{-4\delta}.
\end{equation*}
To see this, note that for each $x \in H$
\begin{align*}
\sum_{k\in\N} \norm{B(x)e_k}{H}^2
&=
\sum_{k\in\N} \abs{\langle x,e_k\rangle_H}^2 \norm{1}{H}^2
=
\norm{x}{H}^2, 
\\
\sum_{k\in\N} \norm{B(x)e_k}{V}^2
&=
\sum_{k\in\N} \abs{\langle x,e_k\rangle_H}^2 \norm{1}{V}^2
=
\norm{x}{H}^2\norm{1}{V}^2,
\end{align*}
and for each $n\in\mathbb N$,
\begin{align*}
  1\geq \sup_{x \in H} \frac{\sum_{k=n}^\infty \norm{B(x)e_k}{H}^2 }{1+\norm{x}{H}^2 }
  &\geq \sup_{\lambda>0}
    \frac{\sum_{k=n}^\infty \norm{B(\lambda e_n)e_k}{H}^2 }{1+\norm{\lambda e_n}{H}^2 } 
    = \sup_{\lambda>0} \frac{\lambda^2}{1+ \lambda^2} = 1, 
  \\
  \sup_{x \in V} \frac{\sum_{k=n}^\infty \norm{B(x)e_k}{H}^2 }{1+\norm{x}{V}^2 } 
  &\leq
    \sup_{x \in V} \frac{\sum_{k=n}^\infty \pi^{-4\delta} k^{-4\delta} \abs{\ip{x}{(-\Delta)^\delta e_k}{H}}^2 }{1+\norm{x}{V}^2 }
  \\&\leq
  \pi^{-4\delta}n^{-4\delta} \sup_{x \in V} \frac{\norm{x}{V}^2 }{1+\norm{x}{V}^2 }
  = \pi^{-4\delta}n^{-4\delta} .
\end{align*}
Thus, in this example Theorem~\ref{thm:weak_errorG} with $H=V=L^2((0,1))$ does not establish convergence. Similiarly, it can be shown that setting $H=V$ equal to the domain of $(-\Delta)^\delta$ does not establish convergence either. However, one obtains a positive rate of convergence by choosing $V$ strictly smaller than $H$. 
\end{example}

\subsection{Sharpness of the rate in the additive noise case}\label{sec:sharpness}

The weak and strong rates provided by Theorem~\ref{thm:weak_errorG} are sharp in the additive noise case when there is no drift, as is shown next. Some related results on spatial discretizations can be found in~\cite{jacobe2018lower}. 

\begin{proposition}\label{prop:additive}
Let $(H,\ip{\cdot}{\cdot}{H})$ and $(U, \ip{\cdot}{\cdot}{U})$ be
separable $\RR$-Hilbert spaces, 
let $(\Omega, \mathcal F,
(\mathcal F_t)_{t \in [0,1]},\mathbb P)$ be a stochastic basis,
let $(W_t)_{t\in [0,1]}$ be an $\Id_U$-cylindrical $(\F_t)$-Wiener process, 
let $B \in \HS(U;H)$ be injective, 
let $(e_k)_{k\in\N}$ be an orthonormal basis of eigenvectors of $B^*B$, 
for each $n\in\N\cup\{\infty\}$ let $P_n \in L(U)$ be the orthogonal projection onto the closure of the linear span of $\{e_k\colon k\in \N\cap [0,n)\}$, 
let $X^n$ be a predictable process which satisfies for all $t \in [0,1]$ that
$[X^n_t]_{\PP,\borel{H}} = \int_0^t B P_n \d W_s$,
and let $\phi=\exp(-\norm{\cdot}{H}^2/2) \in C^2_b(H)$. 
Then
\begin{equation*}
\lim_{n\to\infty} \frac{\norm{X^n_1-X^\infty_1}{L^2(\Omega;H)}^2}{\sum_{k=n}^\infty \norm{Be_k}{H}^2}=1, 
\qquad
\lim_{n\to\infty} \frac{\E{\phi\left(X^n_1\right)} -
\E{\phi\left(X^\infty_1\right)}}{ \sum_{k=n}^\infty \norm{Be_k}{H}^2 } = \frac{\E{\phi(X^\infty_1)}}{2} \in (0,\infty). 
\end{equation*}
\end{proposition}

\begin{proof}
By It\^o's isometry it holds for each $n \in \mathbb N \cup \{\infty\}$ that 
\begin{equation*}
\norm{X^n_1-X^\infty_1}{L^2(\Omega;H)}^2
=
\norm{B(P_n-P_\infty)}{\HS(U;H)}^2
=
\sum_{k=n}^\infty \norm{Be_k}{H}^2. 
\end{equation*}
This proves the strong convergence rate, and it remains to show the weak convergence rate. For each $n \in \N \cup \{\infty\}$ the random variable $X^n_1$ is Gaussian with covariance $BP_nP_n^*B^* \in L_1(H)$ by \cite[Theorem~5.2]{daPrato2014se}. 
Setting $\lambda_k = \norm{Be_k}{H}$, 
one has for each $n \in \N\cup \{\infty\}$ by \cite[Proposition~2.17]{daPrato2014se} and a singular value decomposition of the operator $BP_n \in \HS(U;H)$ that
\begin{align*}
\E{\phi(X^n_1)} 
&=
\exp\left(-\frac12 \trace \big(\log(1+BP_nP_n^*B^*)\big)\right)
\\&=
\exp\left(-\frac12 \trace \big(\log(1+P_n^*B^*BP_n)\big)\right) =
\exp\bigg(-\frac12 \sum_{k=0}^{n-1} \log(1+\lambda_k^2) \bigg). 
\end{align*}
The basic inequalities
\begin{gather*}
\forall x,y\in \R: x\leq y \Rightarrow \exp(y)-\exp(x)\geq \exp(x)(y-x), 
\\
\forall \epsilon>0 \exists \delta>0 \forall x\in [0,\delta]: \log(1+x) \geq (1-\epsilon) x,
\end{gather*}
imply that 
\begin{align*}
\hspace{2em}&\hspace{-2em}
\liminf_{n\to\infty} \frac{\E{\phi\left(X^n_1\right)} -
\E{\phi\left(X^\infty_1\right)}}{ \sum_{k=n}^\infty \lambda_k^2 }
\\&\geq
\liminf_{n\to\infty}  \frac{\exp\left(-\frac12\sum_{k=0}^\infty \log\left(1+\lambda_k^2\right)\right) \frac12\sum_{k=n}^\infty \log\left(1+\lambda_k^2\right) }{\sum_{k=n}^\infty  \lambda_k^2 }
\\&\geq
\sup_{\epsilon>0} \liminf_{n\to\infty}  \frac{\exp\left(-\frac12\sum_{k=0}^\infty \log\left(1+\lambda_k^2\right)\right) \frac12\sum_{k=n}^\infty (1-\epsilon)\lambda_k^2 }{\sum_{k=n}^\infty  \lambda_k^2 }
\\&=
\frac12 \exp\left(-\frac12\sum_{k=0}^\infty \log\left(1+\lambda_k^2\right)\right)
=
\frac{\E{\phi(X^\infty_1)}}{2}
\end{align*}
and
\begin{align*}
\hspace{2em}&\hspace{-2em}
\limsup_{n\to\infty} \frac{\E{\phi\left(X^n_1\right)} -
\E{\phi\left(X^\infty_1\right)}}{ \sum_{k=n}^\infty \lambda_k^2 }
\\&\leq
\limsup_{n\to\infty}  \frac{\exp\left(-\frac12\sum_{k=0}^{n-1} \log\left(1+\lambda_k^2\right)\right) \frac12\sum_{k=n}^\infty \log\left(1+\lambda_k^2\right) }{\sum_{k=n}^\infty  \lambda_k^2 }
\\&\leq
\limsup_{n\to\infty}  \frac{\exp\left(-\frac12\sum_{k=0}^{n-1} \log\left(1+\lambda_k^2\right)\right) \frac12\sum_{k=n}^\infty \lambda_k^2 }{\sum_{k=n}^\infty  \lambda_k^2 }
\\&=
\frac12 \exp\left(-\frac12\sum_{k=0}^\infty \log\left(1+\lambda_k^2\right)\right)
=
\frac{\E{\phi(X^\infty_1)}}{2}.
\qedhere
\end{align*}
\end{proof}

\section{Error analysis}\label{sec:error}

This section contains the proof of our main technical result, Proposition~\ref{prop:weak_error_strong}, where the convergence rate of noise approximations is analyzed in an abstract framework for semilinear stochastic evolution equations with multiplicative noise. 
Proposition~\ref{prop:weak_error_strong} is used in the proof of Theorem~\ref{thm:weak_errorG} and will be applied to various examples in Section~\ref{sec:examples}. 

\subsection{Setting}\label{sec:setting}
We will repeatedly use the following standard setting: 
let $T\in (0,\infty)$,
let $(H,\ip{\cdot}{\cdot}{H})$ and $(U, \ip{\cdot}{\cdot}{U})$ be separable $\RR$-Hilbert spaces, 
let $\mathbb U$ be an orthonormal basis of $U$, 
let $(\Omega, \mathcal F, (\mathcal F_t)_{t \in [0,T]},\mathbb P)$ be a stochastic basis, 
let $(W_t)_{t\in [0,T]}$ be an $\Id_U$-cylindrical $(\F_t)$-Wiener process, 
let $S \in B([0,\infty);L(H))$ be a uniformly bounded and strongly continuous semigroup, 
let $F\in \Lip(H)$, 
let $B\in \Lip(H; \HS(U; H))$, 
and let $\xi \in \cL^2(\Omega;H)$ be $\mathcal F_0\slash\borel{H}$-measurable. 

\subsection{Existence and uniqueness of solutions}\label{sec:existence}

The following lemma collects some results on existence and uniqueness of solutions of stochastic evolution equations, explicit a-priori bounds, and continuous dependence of the solution on the semigroup. 

\begin{lemma}
\label{lem:apriori}
Assume the setting of Section~\ref{sec:setting}. Then the following statements hold true:
\begin{enumerate}[label=(\roman*)]
\item
\label{item:apriori1} 
There exists an up to modifications unique predictable process $X\colon\Omega \times [0,T]\to H$ which satisfies that $\P{\int_0^T \norm{X_t}{H}^2 \d t <\infty} =1$ and for each $t \in [0,T]$, 
\begin{equation*}
\left[X_t\right]_{\PP,\borel{H}} = \left[S_t\xi + \int_0^t S_{t-s}
F(X_s)  \d s\right]_{\PP, \borel{H}} +
\int_0^t S_{t-s} B(X_s) \d W_s,	
\end{equation*}

\item
\label{item:apriori2} 
The process $X$ satisfies 
\begin{multline*}
\norm{X}{B([0,T];L^2(\Omega;H))} \leq \norm{S}{B([0,T];L(H))}\\
\times
\left(\norm{\xi}{L^2(\Omega;H)} +
\sqrt{2T}\norm{F}{\Lip(H)} +
\sqrt{2T}\norm{B}{\Lip(H,
  \HS(U;H))}\right)\\
\times \exp\left(T\norm{S}{B([0,T];L(H))}^2 \left(\frac12
+
\norm{F}{\Lip(H)}^2 + \norm{B}{\Lip(H;\HS(U;H))}^2\right)\right)
\end{multline*}

\item
\label{item:apriori3}
For $n\in\NN$ let $S^n\colon [0,\infty)\to L(H)$ be a strongly continuous semigroup such that for $x \in H$ 
\begin{align*}
\lim_{n\to\infty} \norm{S^nx-Sx}{B([0,T],H)}=0,
&&
\sup_{n\in\NN}\norm{S^n}{B([0,T],L(H))}<\infty,
\end{align*}
and let $X^n$ be the process $X$ given by~\ref{item:apriori1} with $S$ replaced by $S^n$. Then
\begin{equation*}
\lim_{n\to\infty} \norm{X-X^n}{B([0,T];L^2(\Omega;H))} =0.
\end{equation*}
\end{enumerate}  
\end{lemma}

\begin{proof}
\ref{item:apriori1} follows from
\cite[Theorem~7.2]{daPrato2014se} and moreover, we get 
\begin{equation*}
\sup_{t\in[0,T]}\mathbb E\left[ \norm{X_t}{H}^2\right]<\infty.
\end{equation*}
\ref{item:apriori2}: To derive the explicit bound, we follow the proof of
\cite[Lemma~2.1]{jacobe2015weak} and apply the mild It\^o
formula~\cite[Corollary 1]{daprato2010mild} to the transformation
$\phi(x) := \norm{x}{H}^2$, which belongs to $C^2(H;\R)$. This yields
\begin{align*}
\E{\norm{X_t}{H}^2} 
&= 
\E{\norm{S_t \xi}{H}^2} 
+2 \EE\int_0^t \ip{S_{t-s}F(X_s)}{S_{t-s}X_s}{H}  \d s 
\\&\qquad\qquad 
+\EE\int_0^t\norm{S_{t-s}B(X_s)}{\HS(U;H)}^2\d s
\\&\leq 
\E{\norm{S_t\xi}{H}^2} 
+2\sqrt{\EE\int_0^t\norm{S_{t-s}F(X_s)}{H}^2\d s}
\sqrt{\EE\int_0^t \norm{S_{t-s}X_s}{H}^2 \d s} 
\\&\qquad\qquad  
+\EE\int_0^t\norm{S_{t-s}B(X_s)}{\HS(U;H)}^2\d s
\end{align*}
By the inequality of arithmetic and geometric means, which states that all $x,y \in [0,\infty)$ satisfy $\sqrt{xy}\leq (x+y)/2$, it follows that
\begin{align*}
\E{\norm{X_t}{H}^2} 
&\leq 
\norm{S}{B([0,T];L(H))}^2
\left(\E{\norm{\xi}{H}^2}\vphantom{norm{S}{B([0,T];L(H))}^2\E{\int_0^t
\left(\norm{F}{\Lip(H)}^2 
+\norm{B}{\Lip(H;\HS(U;H))}^2\right)\left(1 +\norm{X_s}{H}\right)^2 
+\norm{X_s}{H}^2 \d s}}\right.
\\&\qquad
+ \left.\E{\int_0^t \left(\norm{F}{\Lip(H)}^2 
+\norm{B}{\Lip(H;\HS(U;H))}^2\right)\left(1 +\norm{X_s}{H}\right)^2 
+\norm{X_s}{H}^2 \d s}\right)
\\&\leq 
\norm{S}{B([0,T];L(H))}^2
\left(\norm{\xi}{L^2(\Omega,H)}^2 +
2T\norm{F}{\Lip(H)}^2 +
2T\norm{B}{\Lip(H;\HS(U;H)}^2\right)
\\&\qquad
+ \norm{S}{B([0,T];L(H))}^2 \left(1+2\norm{F}{\Lip(H)}^2 
+ 2\norm{B}{\Lip(H;\HS(U;H))}^2\right) 
\int_0^t \E{ \norm{X_s}{H}^2} \d s.
\end{align*}
By \ref{item:apriori1} all terms are finite. Thus, an application of Gronwall's lemma proves~\ref{item:apriori2}.

\ref{item:apriori3}: For each $n\in\NN$ and $t\in [0,T]$ we get by the triangle inequality, Minkowski inequality, Jensen's inequality, Fubini's theorem and It\^o's isometry that 
\begin{align*}
\hspace{2em}&\hspace{-2em}
\norm{X_t-X^n_t}{L^2(\Omega;H)}
\leq
\norm{(S_t-S^n_t)\xi}{L^2(\Omega;H)}
\\&\qquad
+\norm{\int_0^t(S_{t-s}-S^n_{t-s})F(X_s)\d s}{L^2(\Omega;H)}
+\norm{\int_0^t(S_{t-s}-S^n_{t-s})B(X_s)\d W_s}{L^2(\Omega;H)}
\\&\qquad
+\norm{\int_0^t S^n_{t-s}\left(F(X_s)-F(X^n_s)\right)\d s}{L^2(\Omega;H)}
+\norm{\int_0^t S^n_{t-s}\left(B(X_s)-B(X^n_s)\right)\d W_s}{L^2(\Omega;H)}
\\&\leq
\norm{(S-S^n)\xi}{L^2(\Omega;C([0,T];H))}
\\&\qquad
+\int_0^T\norm{(S-S^n)F(X_s)}{L^2(\Omega;C([0,T];H))}\d s
+\sqrt{\int_0^T\norm{(S-S^n)B(X_s)}{L^2(\Omega;C([0,T];H))}^2\d s}
\\&\qquad
+\norm{S^n_t}{B([0,T];L(H))}
\left(\norm{F}{\Lip(H)}\sqrt{T}+\norm{B}{\Lip(H;\HS(U;H))}\right)
\sqrt{ \int_0^t \norm{X_s-X^n_s}{L^2(\Omega;H)}^2\d s }.
\end{align*}
Taking the square and applying Jensen's inequality yields
\begin{align*}
\hspace{2em}&\hspace{-2em}
\norm{X_t-X^n_t}{L^2(\Omega;H)}^2
\leq
2\Bigg(
\norm{(S-S^n)\xi}{L^2(\Omega;C([0,T];H))}
\\&\qquad
+\int_0^T\norm{(S-S^n)F(X_s)}{L^2(\Omega;C([0,T];H))}\d s
+\sqrt{\int_0^T\norm{(S-S^n)B(X_s)}{L^2(\Omega;C([0,T];H))}^2\d s}
\Bigg) 
\\&\qquad
+2\norm{S^n}{B([0,T];L(H))}^2
\left(\norm{F}{\Lip(H)}\sqrt{T}+\norm{B}{\Lip(H;\HS(U;H))}\right)^2
\int_0^t \norm{X_s-X^n_s}{L^2(\Omega;H)}^2\d s.
\end{align*}
An application of Gronwall's lemma, which is justified by \ref{item:apriori2}, shows that
\begin{align*}
\hspace{2em}&\hspace{-2em}
\norm{X-X^n}{B([0,T],L^2(\Omega;H))}^2
\leq
2\Bigg(
\norm{(S-S^n)\xi}{L^2(\Omega;C([0,T];H))}
\\&\qquad
+\int_0^T\norm{(S-S^n)F(X_s)}{L^2(\Omega;C([0,T];H))}\d s
+\sqrt{\int_0^T\norm{(S-S^n)B(X_s)}{L^2(\Omega;C([0,T];H))}^2\d s}
\Bigg) 
\\&\qquad
\times \exp\bigg(2T\norm{S^n}{B([0,T];L(H))}^2
\left(\norm{F}{\Lip(H)}\sqrt{T}+\norm{B}{\Lip(H;\HS(U;H))}\right)^2\bigg).
\end{align*}
The right-hand side tends to zero as $n\to\infty$ by the dominated convergence theorem. 
\end{proof}

\subsection{Uniform bounds on solutions of Kolmogorov's equation}\label{sec:kolmogorov}
Under suitable regularity conditions the Markovian semigroup associated to a stochastic evolution equation satisfies the Kolmogorov equation. 
This is made precise in the following lemma.
Part~\ref{item:kolmogorov1} of the lemma bounds the spatial derivatives uniformly in the coefficients of the equation, 
and Part~\ref{item:kolmogorov2} establishes spatial and temporal differentiability and the Kolmogorov equation. 
Part~\ref{item:kolmogorov2} is well-known, and we will roughly
follow the idea in~\cite[Theorem 9.23]{daPrato2014se}. We will,
however, go into more details because of some inaccuracies in this
reference.\footnote{For example, the statement $u\in C^{1,2}_b$ is not
  correct if $H=\R$, $F\equiv 0$, $B\equiv 0$, $A = \Id_H$, and $\phi = \sin$ because $\sup_{x\in H}
  \abs{\ddt u(t,x)} = \sup_{x\in H} \abs{\cos(e^tx)x} = \infty$ for all $t\in[0,T]$.} 
An alternative would be to generalize the proof of \cite[Lemma~6.2.(xix)]{hefter2016weak} to our assumptions.

\begin{lemma}
\label{lem:uNuniform}
Assume that the setting of Section~\ref{sec:setting} holds true and,
additionally, 
assume that $S$ has a bounded generator $A\in L(H)$, 
$F\in C^2_b(H)$, 
and $B\in C^2_b(H; \HS(U; H))$, 
for each $x \in H$ let $X^x\colon [0,T] \times \Omega \to H$ 
be a predictable stochastic process which satisfies $\P{\int_0^T \norm{X^x_t}{H}^2\d t<\infty}=1$ and 
for all $t\in [0,T]$,
\begin{equation*}
\left[X_t^x\right]_{\PP,\borel{H}} = \left[ S_t x + \int_0^t S_{t-s}F(X^x_s)
\d s \right]_{\PP,\borel{H}} + \int_0^t S_{t-s}B(X^x_s)\d W_s,
\end{equation*}
let $\phi \in C^2_b(H;\R)$, and for each $t \in [0,T]$ and $x \in H$ let $u(t,x) = \E{\phi(X_{T-t}^x)}$.
Then the following statements hold true:
\begin{enumerate}[label=(\roman*)]
\item
\label{item:kolmogorov1}
For all $t\in [0,T]$ the map $(H\ni x \mapsto u(t,x)\in \R)$ is twice Fr\'echet differentiable and satisfies 
$\ddx u \in C_b([0,T]\times H; L(H;\R))$, $\ddxx u\in
C_b([0,T]\times H; L^{(2)}(H;\R))$, and
\begin{align}
\label{eq:dxuest}
\sup_{t \in [0,T]}\sup_{x \in H}\norm{\ddx u(t,x)}{L(H;\R)}
&\leq \norm{\phi}{C_b^1(H;\R)} C_3,
\\
\label{eq:dxxuest}
\sup_{t \in [0,T]}\sup_{x \in H}\mednorm{\ddxx u(t,x)}{L^{(2)}(H;\R)}
&\leq \norm{\phi}{C_b^2(H;\R)} (C_3^2+ C_4),
\end{align}
where the constants $C_3$ and $C_4$ are given by
\begin{align*}
C_3 &=\norm{S}{B([0,T];L(H))}
      \exp\left(T\norm{S}{B([0,T];L(H))}^2 \left( \norm{ F }{ C_{ {b} }^1( H ) } +
      \frac{1}{2} \norm{ B }{ C_{ {b} }^1( H; \HS( U;H ) )
      }^2 \right)  \right),
\\
C_4 &= \exp\left(T \left(
      \frac72 \norm{F}{C_b^2(H)}+
      4 \norm{ B }{ C_{ {b} }^1( H; \HS( U;H ) )
      }^2 \right)\norm{S}{B([0,T];L(H))}^4\right)   
\\&\qquad
    \times\sqrt{T}\norm{S}{B([0,T];L(H))}^3\sqrt{\norm{F}{C^2_b(H)} +
    2\norm{B}{C^2_b(H;\HS(U;H))}^2}.
\end{align*}

\item
\label{item:kolmogorov2}
The function $u$ is of class $C^{1,2}$ and satisfies for each $t \in [0,T]$ and $x \in H$ that
\begin{equation*}
-(\ddt u)(t,x) = (\ddx u)(t,x) (Ax + F(x)) + \frac12 \sum_{b \in \mathbb U} (\ddxx u)(t,x) (B(x)Pb,B(x)Pb).
\end{equation*}
\end{enumerate}
\end{lemma}

\begin{proof}
\ref{item:kolmogorov1}: 
Note that the generator $A$ of $S$ is bounded and thus sectorial. 
Hence, by \cite[Theorem 3.3.(iii)]{andersson2018regularity} we get
for each $t\in [0,T]$ that the mapping $H\ni x\mapsto u(t,x)\in\R$ is of class $C^2_b(H;\R)$. 
Moreover,~\cite[Theorem~3.3.(v)]{andersson2018regularity} 
yields that for each $t\in [0,T]$ and $x, v, w\in H$, 
\begin{align}
\label{eq:ux}
(\ddx u)(t,x)v &= \E{\phi'(X_{T-t}^{x}) X_{T-t}^{1,(x,v)}},
\\
\label{eq:uxx}
(\ddxx u)(t,x)(v,w) &= \E{\phi''(X_{T-t}^{x})(X_{T-t}^{1,(x,v)},
X_{T-t}^{1,(x,w)}) +
\phi'(X_{T-t}^{x})(X_{T-t}^{2,(x,v,w)})},
\end{align}
where for each $x, v, w\in H$ the stochastic processes
$X^{1,(x,v)},X^{2,(x,v,w)}\colon [0,T]\times \Omega\to H$ are predictable and satisfy that 
$\P{\int_0^T \left(\|X_t^{1,(x,v)}\|_H^2 + \|X_t^{2,(x,v,w)}\|_H^2\right) \d t < \infty}  = 1$
and for each $t\in [0,T]$
\begin{align*}
\left[X_t^{1,(x,v)}\right]_{\PP,\borel{H}} 
&= \left[S_t v +
  \int_0^t S_{t-s} \left( F'(X^x_s)X_s^{1,(x,v)} \right)\d s \right]_{\PP,\borel{H}}
\\&\qquad
+\int_0^t S_{t-s}\left(B'(X_s^{x}) X_s^{2,(x,v,w)}\right) \d W_s, \\
\left[X_t^{2,(x,v,w)}\right]_{\PP,\borel{H}} 
&= \left[\int_0^t S_{t-s}
  \left(F'(X^{x}_s)X_s^{2,(x,v,w)} \right.\right.\\
&\qquad\qquad\qquad \left.\left.+ F''(X^{x}_s)(X_s^{1,(x,v)},X_s^{1,(x,w)})\right)\d s\vphantom{\int_0^t} \right]_{\PP,\borel{H}}
\\
&\qquad+\int_0^t S_{t-s}\left(B'(X_s^{x})
  X_s^{2,(x,v,w)}\right.\\
&\qquad\qquad\qquad \left.+ B''(X_s^{x})(X_s^{1,(x,v)}, X_s^{1,(x,w)})\right) \d W_s.
\end{align*}
Hence, it holds for each $t\in [0,T]$ and $x, v, w\in H$ that
\begin{align}
\label{eq:norm_ux}
\abs{\left(\ddx u\right)(t,x)v} &\leq \norm{\phi}{C^1_b(H,\R)}
\E{\mednorm{X_{T-t}^{1,(x,v)}}{H}},
\\
\label{eq:norm_uxx}
\big|\big(\ddxx u\big)(t,x)(v,w)\big| 
&\leq 
\norm{\phi}{C^2_b(H,\R)}
\E{\mednorm{X_{T-t}^{1,(x,v)}}{H}\mednorm{X_{T-t}^{1,(x,w)}}{H}+ \mednorm{X_{T-t}^{2,(x,v,w)}}{H}}.
\end{align}
From \cite[Theorem~3.3.(vi)]{andersson2018regularity} we get for
each $p\in[1,\infty)$ and $x, v, w\in H$ that
\begin{equation*}
\sup_{t\in[0,T]} \E{\mednorm{X_t^{1,(x,v)}}{H}^{2p}+\mednorm{X_t^{2,(x,v,w)}}{H}^{2p}} <\infty.
\end{equation*}
To derive explicit bounds, we use mild It\^o
calculus and proceed as in the proof of Lemma~\ref{lem:apriori}.\ref{item:apriori2}, 
following the proof of \cite[Lemma~2.1]{jacobe2015weak}. 
To this end, note that for all $p\in[1,\infty)$ 
the function $\psi_p(x) := \norm{x}{H}^{2p}$ is
twice Fr\'echet differentiable with derivatives, for $x$, $v$, $w\in H$,
\begin{align*}
\psi_p'(x)v &= 2p \ip{x}{v}{H} \norm{x}{H}^{2p-2},\\
\psi_p''(x)(v,w) &=
                   \begin{cases}
                     2 \ip{v}{w}{H}, & p =1         \\
                     0,& x=0, p>1\\                     
                     2p\ip{v}{w}{H}\norm{x}{H}^{2p-2} +4p(p-1) \ip{x}{v}{H} \ip{x}{w}{H}
                     \norm{x}{H}^{2p-4},& x\neq 0,p>1\\
                   \end{cases}
\end{align*}
In particular, it holds for each $p\in[1,\infty)$ by the Cauchy--Schwarz inequality that
\begin{align*}
\abs{\psi_p'(x)v} &\leq 2p \norm{v}{H} \norm{x}{H}^{2p-1},
\\
\abs{\psi_p''(x)(v,w)} &\leq 2p \norm{w}{H}\norm{v}{H}\norm{x}{H}^{2p-2}
+ 4p(p-1) \norm{v}{H}\norm{w}{H} \norm{x}{H}^{2p-2}.
\end{align*}
Then we get by the mild It\^o formula \cite[Corollary~1]{daprato2010mild} 
for each $p\in [1,\infty)$ and $x, v\in H$ that
\begin{align*}
& \EE \left[ \mednorm{ X_t^{1,(x, v) } }{H}^{2p} \right] 
= 
\E{ \psi_p\big(X_t^{1,(x,v)}\big)} 
\\& = 
\psi_p \big( S_t v \big) + \int_{0}^{t} \E{
\psi_p' \big( S_{t-s} X_s^{1,(x,v)} \big) S_{t-s}
F'( X_s^x )X_s^{1,(x,v)} } \d s
\\& \qquad 
+ \frac{1}{2} \sum_{ b \in \UU } \int_{0}^{t} \EE\left[
\psi_p'' \big( S_{t-s}X_s^{1,(x,v)} \big) \left( S_{t-s} \big( B'( X_s^x )X_s^{1,(x,v)} \big) b,
\right.\right. 
\\&\qquad\qquad\qquad\qquad\qquad\qquad\qquad\qquad\qquad\qquad
\left.\left. S_{t-s} \big( B'( X_s^x )X_s^{1,(x,v)} \big) b \right) \right] \d s 
\\& \leq 
\norm{ v }{H}^{2p} \norm{S}{B([0,T];L(H))}^{2p} + 2p \norm{S}{B([0,T];L(H))}^{2p} \norm{ F }{ C_{ {b} }^1( H; H ) } \int_{0}^{t} \EE \left[ \mednorm{X_s^{1,(x,v)} }{H}^{2p} \right] \d s 
\\& \qquad 
+ p \norm{S}{B([0,T];L(H))}^{2p} \norm{ B }{ C_{ {b} }^1( H; \HS( U;H ) ) }^2 \int_{0}^{t} \EE \left[ \mednorm{X_s^{1,(x,v)} }{H}^{2p} \right] \d s 
\\& \qquad + 2p(p-1) \norm{S}{B([0,T];L(H))}^{2p} \norm{ B }{ C_{ {b} }^1( H; \HS( U;H ) ) }^2 \int_{0}^{t} \EE \left[ \mednorm{X_s^{1,(x,v)} }{H}^{2p} \right] \d s 
\\& = 
\norm{ v }{H}^{2p} \norm{S}{B([0,T];L(H))}^{2p}  
+ \int_{0}^{t} \EE \left[ \mednorm{X_s^{1,(x,v)} }{H}^{2p} \right] \d s
\\&\qquad\qquad 
\times p \norm{S}{B([0,T];L(H))}^{2p} \left( 2\norm{ F }{ C_{ {b} }^1( H; H ) } + (2p-1) \norm{ B }{ C_{ {b} }^1( H; \HS( U;H ) ) }^2 \right),
\end{align*}
which implies by Gronwall's lemma that
\begin{multline*}
\sup_{0\leq t\leq T}\mednorm{X_t^{1,(x,v)}}{L^{2p}(\Omega; H)}     
\leq
\norm{v}{H} \norm{S}{B([0,T];L(H))}
\\
\times \exp\left(T \left( \norm{ F }{ C_{ {b} }^1( H; H ) } +
  \frac{(2p-1)}{2} \norm{ B }{ C_{ {b} }^1( H; \HS( U;H ) )
  }^2 \right) \norm{S}{B([0,T];L(H))}^{2p}\right).
\end{multline*}
In the special case $p=1$ this shows for all $x,v \in H$ that
\begin{equation}\label{eq:XtN1xv}
\sup_{t \in [0,T]} \mednorm{X_t^{1,(x,v)}}{L^{2}(\Omega; H)} \leq C_3 \norm{v}{H}.
\end{equation}
Plugging this into \eqref{eq:norm_ux} proves \eqref{eq:dxuest}. 
For $X_t^{2,(x,v,w)}$, we need only an $L^2$ estimate, which we get
from the mild It\^o formula \cite[Corollary~1 and Example~2]{daprato2010mild} and the Cauchy--Schwarz inequality: for each $t \in [0,T]$ and $x,v,w \in H$ it holds that
{\allowdisplaybreaks
\begin{align*}
\hspace{2em}&\hspace{-2em}
\E{\mednorm{X_t^{2,(x,v,w)}}{H}^2} 
\\&\leq 
2 \int_0^t\E{ \abs{\ip{S_{t-s} X_s^{2,(x,v,w)}}{S_{t-s} F'(X_s^{x})X_s^{2,(x,v,w)}}{H}}}\d s
\\&\quad 
+2 \int_0^t\E{\abs{ \ip{S_{t-s} X_s^{2,(x,v,w)}}{S_{t-s} F''(X_s^{x})(X_s^{1,(x,v)}, X_s^{1,(x,w)})}{H}}\d s }
\\&\quad
+ \int_0^t \E{\norm{S_{s}\left( B'(X_s^{x}) X_s^{2,(x,v,w)}
+ B''(X_s^{x})(X_s^{1,(x,v)},X_s^{1,(x,w)})\right)}{\HS(U;H)}^2}\d
s 
\\&\leq 
\norm{S}{B([0,T];L(H))}^2\left(2
\norm{F}{C^1_b(H)} \int_0^t \E{\mednorm{X_s^{2,(x,v,w)}}{H}^2}\d
s  \right.
\\* &\qquad\quad 
+ \norm{F}{C^2_b(H)} \int_0^t \E{\mednorm{X_s^{2,(x,v,w)}}{H}^2 
+ \mednorm{X_s^{1,(x,v)}}{H}^2\mednorm{X^{1,(x,w)}_s}{H}^2} \d s 
\\* &\qquad\quad 
+ 2\norm{B}{C^1_b(H;\HS(U;H))}^2 \int_0^t\E{\mednorm{X_s^{2,(x,v,w)}}{H}^2}\d s
\\* &\qquad\quad 
+\left. 2\norm{B}{C^2_b(H;\HS(U;H))}^2 
\int_0^t\E{\mednorm{X_s^{1,(x,v)}}{H}^2\mednorm{X_s^{1,(x,w)}}{H}^2}\d s\right)
\\&\leq 
\norm{S}{B([0,T];L(H))}^2 t(\norm{F}{C^2_b(H)} +
2 \norm{B}{C^2_b(H;\HS(U;H))}^2)\\
&\qquad\qquad\qquad\qquad\times
\sup_{0\leq s\leq t}\mednorm{X_s^{1,(x,v)}}{L^4(\Omega;H)}^2\mednorm{X_s^{1,(x,w)}}{L^4(\Omega;H)}^2\\
&\quad\!+\!\norm{S}{B([0,T];L(H))}^2 \left( 3\norm{F}{C^2_b(H)}
\!+\! 2\norm{B}{C^1_b(H;\HS(U;H))}^2\!\right) \int_0^t \E{\mednorm{X_s^{2,(x,v,w)}}{H}^2}\!\d s .
\end{align*}
}
Hence, Gronwall's lemma shows for all $x,v,w \in H$ that
\begin{equation}
\label{eq:XtN2xvw}
\sup_{t \in [0,T]} \mednorm{X_t^{2,(x,v,w)}}{L^2(\Omega;H)} 
\leq C_4 \norm{v}{H}\norm{w}{H}.
\end{equation}
Inserting the bounds \eqref{eq:XtN1xv} and \eqref{eq:XtN2xvw} into
\eqref{eq:norm_uxx} proves \eqref{eq:dxxuest}.

It remains to verify the continuity claims of the statement. 
By \cite[Theorem 2.1.(vii)]{andersson2017differentiability}
it holds for each $p\in (2,\infty)$, $t \in [0,T]$, and $x \in H$ that
\begin{equation}\label{eq:equicont}
\begin{aligned}
\lim_{y\to x} \sup_{t\in [0,T]} \norm{X_t^{x} -
  X_t^{y}}{L^p(\Omega;H)} &= 0,\\
\lim_{y\to x} \sup_{t\in [0,T]}\mednorm{X_t^{1,(x,.)} -
  X_t^{1,(y,.)}}{L(H;L^p(\Omega;H))} &= 0,\\
\lim_{y\to x} \sup_{t\in [0,T]}\mednorm{X_t^{2,(x,\cdot,\cdot)} -
  X_t^{2,(y,\cdot,\cdot)}}{L^{(2)}(H;L^p(\Omega;H))} &= 0.
\end{aligned}
\end{equation}
Moreover, for each $x,v,w \in H$ the processes $X^x$, $X^{1,(x,v)}$, and $X^{2,(x,v,w)}$ admit continuous modifications. 
Thus, Burkholder--Davis--Gundy type inequalities and the dominated convergence theorem yield for each $p\in (2,\infty)$ and $x \in H$ that the mapping $[0,T]\ni t\mapsto X_t^{x} \in L^{p}(\Omega;H)$ is continuous; 
see also \cite[Proof of Lemma 6.2.(xiii)]{hefter2016weak}. 
Similarly, the proof of \cite[Lemma 6.2.(xiv), p 27]{hefter2016weak} shows for each $p \in (2,\infty)$ and $x,v,w \in H$ that the mappings $[0,T]\ni t\mapsto X_t^{1,(x, \cdot)} \in L(H;L^{p}(\Omega;H))$ and $[0,T]\ni t\mapsto X_t^{2,(x,\cdot,\cdot)} \in L^{(2)}(H;L^{p}(\Omega;H))$ are continuous. 
This and~\eqref{eq:equicont} implies for all $p\in (2,\infty)$, $t \in [0,T]$ and $x \in H$ that
\begin{equation}\label{eq:X_cont_xt}
\begin{aligned}
\lim_{s\to t, y\to x} \norm{X_t^x -
  X_s^y}{L^p(\Omega;H)} &= 0,\\
\lim_{s\to t, y\to x} \mednorm{X_t^{1,(x,\cdot)} -
  X_s^{1,(y,\cdot)}}{L(H;L^p(\Omega;H))} &= 0,\\
\lim_{s\to t, y\to x} \mednorm{X_t^{2,(x,\cdot,\cdot)} -
  X_s^{2,(y,\cdot,\cdot)}}{L^{(2)}(H;L^p(\Omega;H))} &= 0.
\end{aligned}
\end{equation}
Thanks to the continuity and boundedness of $\phi'$ and $\phi''$, the subsequence criterion, and the dominated convergence theorem, this implies for all $p\in (2,\infty)$, $t \in [0,T]$ and $x \in H$ that
\begin{align*}
\lim_{\substack{s\to t\\y\to x}}
\norm{\phi' \circ X_t^x -\phi'\circ X_s^y}{L^p(\Omega;L(H;\R))} 
=
\lim_{\substack{s\to t\\y\to x}}
\norm{\phi'' \circ X_t^x -  \phi''\circ X_s^y}{L^p(\Omega;L^{(2)}(H;\R))} = 0.
\end{align*}
Now, for each $p\in (2,\infty)$, $t \in [0,T]$, and $x\in H$ it holds by the triangle inequality, H\"older inequality,~\eqref{eq:ux} and~\eqref{eq:X_cont_xt},
\begin{align*}
\hspace{2em}&\hspace{-2em}
\lim_{\substack{s\to t\\y\to x}}\sup_{\substack{v\in H\\ \norm{v}{H} \leq 1}}\abs{(\ddx u)(T-t,x)v - (\ddx u)(T-s,y)v}
\\& \leq 
\lim_{\substack{s\to t\\y\to x}}\norm{\phi' \circ X_t^x- \phi' \circ
X_s^y}{L^{p}(\Omega;L(H;\R))} \sup_{\substack{v\in H\\ \norm{v}{H} \leq 1}}\mednorm{X_t^{1,(x,v)}}{L(H;L^{\frac{p}{p-1}}(\Omega;H))}\\
&\qquad+
\lim_{\substack{s\to t\\y\to x}}\norm{\phi}{C^1_b(H;\R)}
\sup_{\substack{v\in H\\ \norm{v}{H}
\leq 1}}\mednorm{X_t^{1,(x,v)} -
X_s^{1,(y,v)}}{L^1(\Omega;H)}=0.
\end{align*}
For the second derivative one obtains similarly for each $t \in [0,T]$ and $x \in H$ that
\begin{align*}
\hspace{2em}&\hspace{-2em}
\lim_{\substack{s\to t\\y\to x}}\sup_{\substack{v,\,w\in H\\ \norm{v}{H},\norm{w}{H}\leq1}}\mednorm{\phi'(X_t^x)X_t^{2,(x,v,w)} -\phi'(X_s^y)X_s^{2,(y,v,w)}}{L^1(\Omega;\mathbb R)}
\\&\leq 
\lim_{\substack{s\to t\\y\to x}} \norm{\phi'\circ X_t^x-
\phi' \circ X_s^y}{L^{p}(\Omega;L(H;\R)))}
\mednorm{X_t^{2,(x,\cdot,\cdot)}}{L^{(2)}(H;L^{\frac{p}{p-1}}(\Omega;H))}
\\&\qquad
+ \lim_{\substack{s\to t\\y\to x}} \norm{\phi'}{C_b(H;\R)} 
\mednorm{X_t^{2,(x,\cdot,\cdot)} -
X_s^{2,(y,\cdot,\cdot)}}{L^{(2)}(H;L^p(\Omega;H))}=0,
\end{align*}
and, for $p$, $p_0$, $p_1$, $p_2\in (2,\infty)$ with
$\frac1{p_0}+\frac1{p_1}+\frac1{p_2} = 1$,
\begin{align*}
\hspace{2em}&\hspace{-2em}
\lim_{\substack{s\to t\\y\to x}}\sup_{\substack{v,\,w\in H\\
\norm{v}{H},\norm{w}{H}\leq1}}\mednorm{\phi''(X_t^x)(X_t^{1,(x,v)},
X_t^{1,(x,w)}) - \phi''(X_{s}^{x})(X_{s}^{1,(x,v)},
X_{s}^{1,(x,w)})}{L^1(\Omega;\mathbb R)}
\\&\leq 
\lim_{\substack{s\to t\\y\to x}}
  \mednorm{\phi''\circ X_t^x - \phi''\circ X_s^y}{L^{p_0}(\Omega;L^{(2)}(H;\R))}
\mednorm{X_t^{1,(x,\cdot)}}{L(H;L^{p_1}(\Omega;H))} 
\mednorm{X_t^{1,(x,\cdot)}}{L(H;L^{p_2}(\Omega;H))}
\\&
+\lim_{\substack{s\to t\\y\to x}}\norm{\phi''}{C_b(H;L^{(2)}(H;\mathbb R))}
\mednorm{X_t^{1,(x,\cdot)} - X_s^{1,(y,\cdot)}}{L(H;L^{p}(\Omega;H))} \mednorm{X_t^{1,(x,\cdot)}}{L(H;L^{\frac{p}{p-1}}(\Omega;H))} 
\\&
+\lim_{\substack{s\to t\\y\to x}}\norm{\phi''}{C_b(H;L^{(2)}(H;\mathbb R))}
\mednorm{X_s^{1,(y,\cdot)}}{L(H;L^{p}(\Omega;H))}
\mednorm{X_t^{1,(x,\cdot)}-X_s^{1,(y,\cdot)}}{L(H;L^{\frac{p}{p-1}}(\Omega;H))},
\end{align*} 
which is actually $0$. Plugging this into~\eqref{eq:uxx} yields
\begin{equation*}
\lim_{\substack{s\to t\\y\to x}}
\sup_{\substack{v,\,w\in H\\\norm{v}{H},\norm{w}{H}\leq1}}  
\abs{\ddxx u(T-t,x)(v,w) - \ddxx u(T-s, y)(v,w)} = 0.
\end{equation*}
Thus, we have shown the continuity of $\ddx u$ and $\ddxx u$. This proves \ref{item:kolmogorov1}.
  
\ref{item:kolmogorov2}: We follow the idea in~\cite[Theorem~9.23]{daPrato2014se}.
By \cite[Lemma 6.1.(ii)]{hefter2016weak} it holds for all $t, h\in [0,T]$ with $t+h\leq T$ that
\begin{equation*}
u(t,x) = \E{\phi(X_{T-t}^x)} = \E{u(t+h, X_h^x)}.
\end{equation*}
Fix now $t_0\in (0,T]$. 
As $A\in L(H)$, it holds for all $t\in [0,T]$ and $x\in H$ that
\begin{equation*}
\left[X_t^x\right]_{\PP;\borel{H}} = \left[x+ \int_0^t A X_s^x +
F(X_s^x)\d s \right]_{\PP;\borel{H}}
+\int_0^t B(X_s^x)\d W_s.
\end{equation*}
By the classical (as opposed to mild) It\^o formula it holds for all $h\in [0,t_0]$ that
\begin{equation}\label{eq:ut0hdiff}\begin{aligned}
u(t_0-h, x)-u(t_0,x) &= \E{u(t_0, X_{h}^x) - u(t_0,x)} \\
&=\EE\int_0^h (\ddx u(t_0,X_s^x)) \left( A X_s^x + F(X_s^x)
\right)\d s 
\\&\qquad
+ \frac12 \EE\int_0^h \sum_{b \in \UU}(\ddxx u(t_0,X_s^x))(X_s^x) \big(B(X_s^x)b,B(X_s^x)b\big)\d s.
\end{aligned}\end{equation}
Indeed, the process $(\int_0^t (\ddx u)(t_0,X^x_s)B(X_s^x)\d W_s)_{t\in[0,T]}$ is a true martingale thanks to
Lemma~\ref{lem:apriori}.\ref{item:apriori2} and the estimate 
\begin{align*}
\hspace{2em}&\hspace{-2em}
              \E{\int_0^T \norm{(\ddx u)(t_0,X^x_t)  B(X^x_t)}{\HS(U;H)}^2 \d t}
  \\&\leq
      T\norm{\ddx u}{C_b([0,T]\times H;L(H;\R))}^2 \sup_{t\in[0,T]} \E{ \norm{ B( X^x_t)}{\HS(U;H)}^2 }
  \\&\leq
      T \norm{\ddx u}{C_b([0,T]\times H;L(H;\R))}^2 
 \norm{B}{C^1_b(H;\HS(U;H))}^2 \sup_{t\in[0,T]} \E{\left(1+\norm{ X^x_t}{H}\right)^2} < \infty.
\end{align*}
The integrands under the time integral in~\eqref{eq:ut0hdiff} are continuous because for each $x \in H$ and $p \in (2,\infty)$ the mapping $[0,T]\ni t\mapsto X_t^x\in
L^p(\Omega; H)$ is continuous, 
$u(t_0,\cdot)\in C^2_b(H)$, 
and the mappings $A$, $F$, and $B$ are Lipschitz continuous. 
Thus, the mean value theorem for integration gives
\begin{align*}
\big(\tfrac{\partial^-\!\!}{\partial t\,} u\big)(t_0,x) 
&:= -\lim_{h\searrow 0}\frac1{h} \left(u(t_0-h,x) - u(t_0,x)\right) \\
&=-(\ddx u)(t_0,x)\left(Ax + F(x) \right)- \frac12 \sum_{b\in \UU} (\ddxx u)(t_0,x)\left( B(x) b, B(x)b\right).
\end{align*}
As the right hand side is continuous in $t_0\in [0,T]$ by part
\ref{item:kolmogorov1}, we get for all $x\in H$ continuity of the mapping $[0,T]\ni t\mapsto \big(\tfrac{\partial^-\!\!}{\partial t\,} u\big)(t,x) \in \R$.
Summarizing, for each $x\in H$, the mapping $[0,T]\ni t\mapsto u(t,x)\in \R$ is left-sided differentiable on $(0,T]$ with continuous left derivative on $[0,T]$. 
Therefore, it is continuously differentiable on $[0,T]$ with $\ddt u = \tfrac{\partial^-\!\!}{\partial t\,} u$.
\end{proof}

\subsection{Yosida approximations}\label{sec:yosida}
For later usage we summarize some properties of Yosida approximations. 
The lemma is formulated under a uniform boundedness assumption on the semigroup, which leads to a simpler estimate in \ref{item:YosidaApprox2}. 
In the context of semilinear stochastic evolution equations this assumption can always be guaranteed by adding and subtracting a multiple of the identity to the generator of the semigroup and the nonlinear part of the drift, respectively.
\begin{lemma}\label{lem:YosidaApprox}
Let $(E, \norm{\cdot}{E})$ be an $\R$-Banach space 
and let $A\colon \dom(A) \subseteq E \to E$ be the generator of a uniformly bounded and strongly continuous semigroup $S$. Then the following statements hold true. 
\begin{enumerate}[label=(\roman*)]
\item\label{item:YosidaApprox1}
The interval $(0,\infty)$ is contained in the resolvent set of $A$, and for each $\lambda \in (0,\infty)$ the bounded linear operator $A_\lambda:= A(\Id_E - A/\lambda)^{-1}$
generates a uniformly continuous semigroup $S^\lambda$.

\item\label{item:YosidaApprox2}
The semigroups $S^\lambda$ satisfy for each $x\in E$ and $T>0$ that
\begin{align*}
\sup_{\lambda \in (0,\infty)} \norm{S^\lambda}{B([0,\infty);L(E))} 
&\leq
\norm{S}{B([0,\infty); L(E))}, 
&
\lim_{\lambda\to \infty} \sup_{t \in [0, T]} \norm{S_tx - S_t^\lambda x}{E} &= 0.
\end{align*}
\end{enumerate}
\end{lemma}

\begin{proof}
\ref{item:YosidaApprox1}: As $S$ is strongly continuous and of negative exponential type, $S$ belongs to $B([0,\infty);L(E))$. By the Feller--Miyadera--Phillips Theorem the interval $(0,\infty)$ is contained in the resolvent set of $A$. Thus, the resolvent $R_\lambda(A):=(\lambda\Id_E-A)^{-1} \in L(E)$ and $A_\lambda \in L(E)$ are well-defined. 

\ref{item:YosidaApprox2}: The Feller--Miyadera--Phillips Theorem comes with the resolvent estimate
\begin{equation*}
\norm{R_{\lambda}(A)^k}{L(E)} \leq
\frac{\norm{S}{B([0,\infty);L(E))}}{\lambda^k},\qquad \forall\,k\in \N,\;\lambda >0,
\end{equation*}
which implies for each $\lambda\in(0,\infty)$ and $t \in [0,\infty)$ that
\begin{align*}
\norm{S_t^\lambda}{L(E)} 
&= 
e^{-\lambda t} \norm{e^{t\lambda^2 R_\lambda(A)}}{L(E)} 
\leq 
e^{-\lambda t} \sum_{k=0}^\infty \frac{t^k\lambda^{2k}}{k!} \norm{R_\lambda(A)^{k}}{L(E)} 
\\&\leq 
e^{-\lambda t} \sum_{k=0}^\infty \frac{t^k\lambda^{2k}}{k!} 
\frac{\norm{S}{B([0,\infty);L(E))}}{\lambda^k}
= 
\norm{S}{B([0,\infty);L(E))}.
\end{align*}  
To see the second statement, note that $\lim_{\lambda\to\infty} A_\lambda x=Ax$, for all $x\in \dom(A)$. 
The Trotter--Kato approximation theorem then implies $\lim_{\lambda\to\infty}S^\lambda_tx=S_tx$ locally uniformly in $t$; 
see e.g.\@ \cite[Lemma II.3.4.(ii) and Theorem III.4.8]{engel1999one}
for details.
\end{proof}

\subsection{Weak error}
We are now ready to present and prove our main technical result, which is an upper bound on the weak error under perturbations of the noise coefficient. 
The proof builds on the results of Sections~\ref{sec:existence}, \ref{sec:kolmogorov},  and~\ref{sec:yosida}.

\begin{proposition}\label{prop:weak_error_strong}
Assume that the setting of Section~\ref{sec:setting} holds true, and 
additionally, 
let $V$ be a separable $\RR$-Hilbert space, which is densely and continuously embedded in $H$, 
let $F \in C^2_b(H)$, 
let $B,\smashedtilde B \in C^2_b(H;\HS(U;H))$, 
let $S|_V\colon[0,\infty)\to L(V)$ be a uniformly bounded and strongly continuous semigroup,
let $F|_V\in \Lip(V)$, 
let $B|_V, \smashedtilde B|_V \in \Lip(V;\HS(U;V))$, 
let $\xi \in \cL^2(\Omega;V)$ be $\mathcal F_0\slash\borel{V}$-measurable, 
and let $X,\smashedtilde X\colon[0,T]\times\Omega\to H$ be predictable processes which satisfy for each $t \in [0,T]$ that
$\P{\int_0^T \left( \norm{X_t}{H}^2 + \norm{\smashedtilde X_t}{H}^2 \right) \d t < \infty} =1$
and 
\begin{align*}
[X_t]_{\PP,\borel{H}} 
&= 
\left[ S_t \xi + \int_0^t S_{t-s}F(X_s)\d s \right]_{\PP,\borel{H}} 
+ \int_0^t S_{t-s}B(X_s)\d W_s,
\\
[\smashedtilde X_t]_{\PP,\borel{H}} 
&= 
\left[ S_t \xi + \int_0^t S_{t-s}F(\smashedtilde X_s)\d s \right]_{\PP,\borel{H}} 
+ \int_0^t S_{t-s}\smashedtilde B(\smashedtilde X_s)\d W_s.
\end{align*}
Then
\begin{equation*}
\abs{\EE\big[\phi\left(X_T\right)-\phi(\smashedtilde X_T)\big]}
\leq 
C \norm{\phi}{C_b^2(H;\R)}\!\sup_{x \in V}\!\frac{\sum_{b \in \UU}\|B(x)b+\smashedtilde B(x)b\|_H \|B(x)b-\smashedtilde B(x)b\|_H }{1+\norm{x}{V}^2}\!,\!
\end{equation*}
where $C =\frac{T}{2}C_1\left(1+C_2^2\right)$ and
\begin{align*}
C_1 
&= 
\exp\left(T \left(\frac72 \norm{F}{C_b^2(H)}+
  4 \norm{ B }{ C_{ {b} }^1( H; \HS( U;H ) )
  }^2 \right)\norm{S}{B([0,\infty);L(H))}^4\right)
\\&\qquad\times
\sqrt{T}\norm{S}{B([0,\infty);L(H))}^3\sqrt{\norm{F}{C^2_b(H)} +
2\norm{B}{C^2_b(H;\HS(U;H))}^2}
\\&\qquad 
+\norm{S}{B([0,\infty);L(H))}^2
\exp\left(T \left( 2\norm{ F }{ C_{ {b} }^1( H) } +
\norm{ B }{ C_{ {b} }^1( H; \HS( U;H ) )
}^2 \right)
\norm{S}{B([0,\infty);L(H))}^2\right)
\\
C_2 
&=
\norm{S}{B([0,\infty), L(V))} \left(\norm{\xi}{L^2(\Omega,V)} +
\sqrt{2T}\norm{F}{\Lip(V)} +
\sqrt{2T}\norm{B}{\Lip(;
\HS(U;V)}\right)
\\&\qquad 
\times \exp\left(T\norm{S}{B([0,\infty); L(V))}^2 \left(\frac12
+ \norm{F}{\Lip(V)}^2 + \norm{B}{\Lip(V;\HS(U;V))}^2\right)\right).
\end{align*}
\end{proposition}

\begin{proof}
We prove the statement in two steps. 

\emph{Step 1.} We assume temporarily that $S\colon[0,\infty)\to L(H)$ and $S|_V\colon[0,\infty)\to L(V)$ are uniformly continuous. The generator of $S$ is denoted by $A \in L(H)$. Let $u \in C^{1,2}([0,T]\times H;\R)$ be defined as in Lemma~\ref{lem:uNuniform}. Then
\begin{equation*}
\E{\phi(\smashedtilde X_T)-\phi(X_T)}
=
\E{u(T,\smashedtilde X_T)-u\left(0,\smashedtilde X_0\right)}.
\end{equation*}
As $A$ is bounded, $\smashedtilde X$ can be written in strong form as
\begin{equation*}
[\smashedtilde X_t]_{\PP,\borel{H}} = \left[\smashedtilde X_0 + \int_0^t (A
\smashedtilde X_s+F(\smashedtilde X_s))\d s \right]_{\PP,\borel{H}} + \int_0^t \smashedtilde B(\smashedtilde X_s) \d W_s,
\end{equation*}
and we get by It\^o's formula and the Kolmogorov equation that
\begin{align*}
\hspace{2em}&\hspace{-2em}
u(T,\smashedtilde X_T) - u(0,\smashedtilde X_0)
\\&
= \int_0^T (\ddt u)(t,\smashedtilde X_t) \d t+ \int_0^T (\ddx u)(t,\smashedtilde X_t) \d
\smashedtilde X_t + \frac12 \int_0^T (\ddxx u)(t,\smashedtilde X_t) \d\, [\smashedtilde X,\smashedtilde X]_t
\\&= 
\int_0^T (\ddx u)(t,\smashedtilde X_t) \smashedtilde B(\smashedtilde X_t) \d W_t+ \frac12 \int_0^T \sum_{b \in \UU}(\ddxx u)(t,\smashedtilde X_t) (\smashedtilde B(\smashedtilde X_t)b,\smashedtilde B(\smashedtilde X_t)b)\d t 
\\&\qquad 
- \frac12  \int_0^T \sum_{b \in \UU}(\ddxx u)(t,\smashedtilde X_t) (B(\smashedtilde X_t)b,B(\smashedtilde X_t)b)\d t 
\\&= 
\int_0^T (\ddx u)(t,\smashedtilde X_t) \smashedtilde B(\smashedtilde X_t) \d W_t 
\\&\qquad 
+ \frac12 \int_0^T \sum_{b \in \UU}(\ddxx u)(t,\smashedtilde X_t) \big(B(\smashedtilde X_t)b+\smashedtilde B(\smashedtilde X_t)b,B(\smashedtilde X_t)b-\smashedtilde B(\smashedtilde X_t)b\big)\d t.
\end{align*}
The stochastic integral on the right-hand side above is a martingale thanks to Lemma~\ref{lem:apriori}.\ref{item:apriori2}, Lemma~\ref{lem:uNuniform}, and the estimate 
\begin{align*}
\hspace{2em}&\hspace{-2em}
\E{\int_0^T \norm{(\ddx u)(t,\smashedtilde X_t) \smashedtilde B(\smashedtilde X_t)}{\HS(U;H)}^2 \d t}
\\&\leq
T \sup_{t\in[0,T]} \E{\norm{(\ddx u)(t,\smashedtilde X_t)}{L(H;\R)}^2 \norm{\smashedtilde B(\smashedtilde X_t)}{\HS(U;H)}^2 }
\\&\leq
T \norm{\ddx u}{C_b([0,T]\times H;\R)}^2 \norm{\smashedtilde B}{C^1_b(H;\HS(U;H))}^2 \sup_{t\in[0,T]} \E{\left(1+\norm{\smashedtilde X_t}{H}\right)^2} < \infty.
\end{align*}
Discarding the martingale part and using that $\smashedtilde X$ takes values in $V$ by Lemma~\ref{lem:apriori}, we can estimate similarly
\begin{align*}
&
\abs{\E{u(T,\smashedtilde X_T) - u(0,\smashedtilde X_0)}}
\\&\leq
\frac{T}{2} \sup_{t\in[0,T]} \E{ \norm{(\ddxx u)(t,\smashedtilde X_t)}{L^{(2)}(H;\R)} \sum_{b \in \UU} \norm{B(\smashedtilde X_t)b+\smashedtilde B(\smashedtilde X_t)b}{H} \norm{B(\smashedtilde X_t)b-\smashedtilde B(\smashedtilde X_t)b}{H} }
\\&\leq
\frac{T}{2} \norm{\ddxx u}{C_b([0,T]\times H;\R)} \left(1+\norm{\smashedtilde X}{B([0,T];L^2(\Omega;V))}^2\right) 
\\&\qquad
\times\sup_{x \in V}\frac{\sum_{b \in \UU}\norm{B(x)b+\smashedtilde B(x)b}{H}\norm{B(x)b-\smashedtilde B(x)b}{H}}{1+\norm{x}{V}^2}.
\end{align*}
Plugging in the estimates for $\smashedtilde X$ and $u$ of Lemmas~\ref{lem:apriori} and \ref{lem:uNuniform} shows the statement of the lemma in the special case where $S\colon[0,\infty)\to L(H)$ and $S|_V\colon[0,\infty)\to L(V)$ are uniformly continuous semigroups. 

\emph{Step 2.} We next show the lemma in the general case where $S\colon[0,\infty)\to L(H)$ and $S|_V\colon[0,\infty)\to L(V)$ are strongly continuous semigroups. 
Let $S^\lambda$ and $S^\lambda|_V$ be the Yosida approximations of $S$ and $S|_V$ constructed in Lemma~\ref{lem:YosidaApprox}, 
and let $(X^\lambda,\smashedtilde X^\lambda,u^\lambda,C^\lambda)$ be defined as $(X,\smashedtilde X,u,C)$ with $S$ replaced by $S^\lambda$ and $A$ replaced by $A_\lambda$. 
Then $\lim_{\lambda\to \infty}\norm{X^\lambda_T-X_T}{L^2(\Omega;H)}=0$ by Lemma~\ref{lem:YosidaApprox}.\ref{item:YosidaApprox2} and Lemma~\ref{lem:apriori}.\ref{item:apriori3}. 
Thus, it follows from the Lipschitz continuity of $\phi\colon H\to \R$ that
\begin{align*}
\E{\phi\left(X_T\right)-\phi\left(\smashedtilde X_T\right)}
&= 
\lim_{\lambda\to \infty} \E{\phi\left(X^\lambda_T\right)-\phi\left(\smashedtilde X^\lambda_T\right)}.
\end{align*}
Therefore, using the result of Step 1, 
\begin{multline*}
\abs{\E{\phi\left(X_T\right)-\phi\left(\smashedtilde X_T\right)}}
\leq 
\limsup_{\lambda\to \infty} C^\lambda \norm{\phi}{C_b^2(H;\R)}\sup_{x \in V}\frac{\sum_{b \in \UU}\norm{B(x)b+\smashedtilde B(x)b}{H}\norm{B(x)b-\smashedtilde B(x)b}{H}\!}{1+\norm{x}{V}^2}.
\end{multline*}
It follows from Lemma~\ref{lem:YosidaApprox} that $\limsup_{\lambda\to\infty} C^\lambda \leq C$, which proves the lemma. 
\end{proof}

\section{Examples}\label{sec:examples}
The examples in this section demonstrate that Theorem~\ref{thm:weak_errorG} is applicable to a wide variety of semilinear stochastic evolution equations. 
Beyond the examples treated below, it can also be applied to the stochastic heat equation with spatially colored noise. 
The convergence rate can be obtained as for the stochastic Schr\"odinger or linearized Korteweg--de-Vries equation below. 
However, the stochastic heat equation with space-time white noise requires different techniques which hinge on the analyticity of the heat semigroup. 

\subsection{Stochastic wave equation}

For the stochastic wave equation with additive noise, weak error rates of temporal discretizations were studied in \cite{hausenblas2010weak, kovacs2013weak} and of spatial discretizations in \cite{cox2017weak, kovacs2012weak, kovacs2015weak}. 
For the stochastic wave equation with multiplicative noise, weak error rates of temporal discretizations were studied in \cite{cox2017weak} and of spatial discretizations in \cite{jacobe2015weak}. 
See \autoref{tab:rates} for a summary of the obtained convergence rates.
We now complement these results by providing an essentially sharp weak convergence rate for noise discretizations of the stochastic wave equation with additive and multiplicative noise.

\begin{proposition}\label{prop:wave}
Let $\theta\in (0,\infty)$, 
let $\epsilon\in (0,1)$, 
let $\rho = (1-\epsilon)/4$, 
let $\sigma\in(1/4,\infty)$, 
let $H$ be the $\R$-Hilbert space $L^2((0,1);\R)$, 
let $\Delta\colon H^2((0,1))\cap H^1_0((0,1))\subset H\to H$ be the Laplace operator with Dirichlet boundary conditions on $H$, 
let $A= \theta \Delta$, 
let $(H_r)_{r \in \R}$ be a family of interpolation spaces
associated to $-A$, 
let one of the following two statements hold true, 
\begin{enumerate}[label=(\alph*)]
\item\label{item:wavea} 
$\eta=0$, $f_0 \in H_{\rho- \frac12}$, $f_1 \in L^{\infty}((0,1))$, and $f\colon (0,1)\times\R\ni(s,x)\mapsto f_0(s)+f_1(s) x\in\R$,

\item\label{item:waveb} 
$\eta \in (0,\rho) \cap (0,1/4-\rho)$ and $f \in C^{0,2}_b([0,1]\times\R;\R)$
\end{enumerate}
let $b_0\in H^{2\rho}((0,1))$, 
let $b_1 \in H^{2\sigma}((0,1))$,
let $\mathbf{H} = H_\eta \times H_{\eta-1/2}$, 
let $\mathbf{A}\colon D(\mathbf{A})\subset \mathbf{H} \to \mathbf{H}$ be the
linear operator which satisfies $D(\mathbf{A})=H_{\eta+1/2}\times H_\eta$ and $[\forall (x_1,x_2) \in D(\mathbf{A}):\mathbf{A}(x_1,x_2)=(x_2,Ax_1)]$, 
let $\mathbf{V}=H_\rho \times H_{\rho-1/2}$, 
let $U=H$, 
let $(\Omega,\mathcal F,(\mathcal F_t)_{t\in [0,T]},\mathbb P)$ be a stochastic basis, 
let $(W_t)_{t\in [0,T]}$ be an $\Id_U$-cylindrical $(\F_t)_{t\in [0,T]}$-Wiener process, 
let $\boldsymbol \xi \in \cL^2(\Omega;\mathbf V)$ be $\F_0\slash
\borel{\mathbf V}$-measurable,
for each $k \in \N$ let $e_k =(\sqrt{2}\sin(k\pi s))_{s \in (0,1)} \in U$, 
and for each $n \in \N\cup\{\infty\}$ let $P_n\in L(U)$ be the orthogonal projection onto the closure of $\operatorname{span}_{\R}\{e_k: k \in \N \cap [0,n)\}$. 
Then the following statements hold:
\begin{enumerate}[label=(\roman*)]
\item\label{item:wave1} $\mathbf{A}$ is the generator of a strongly continuous semigroup $\mathbf{S}\colon[0,\infty)\to L(\mathbf{H})$, which restricts to a strongly continuous semigroup $\mathbf{S}|_\mathbf{V}\colon [0,\infty)\to L(\mathbf{V})$.

\item\label{item:wave2} There are unique mappings $\mathbf{F}\in C^2_b(\mathbf{H})$ and $\mathbf{B} \in C^2_b(\mathbf{H};\HS(U;\mathbf{H}))$ which satisfy for all $(x_1,x_2) \in \mathbf{H}$, $u \in U$, and $s \in [0,1]$ that 
\begin{equation*}
\mathbf{F}(x_1,x_2)(s) = (0,f(s,x_1(s))), \qquad (\mathbf{B}(x_1,x_2)u)(s) = (0,b_0(s) + b_1(s) x_1(s)u(s)).
\end{equation*}
The mappings $\mathbf{F}$ and $\mathbf{B}$ restrict to $\mathbf{F}|_{\mathbf{V}} \in \Lip(\mathbf{V})$ and $\mathbf{B}|_{\mathbf{V}} \in \Lip(\mathbf{V};\HS(U;\mathbf{V}))$.

\item\label{item:wave3} For each $n \in \N \cup \{\infty\}$ there is an up to modifications unique predictable process $\mathbf X^n\colon [0,T]\times \Omega\to \mathbf{H}$ which satisfies that $\P{\int_0^T \norm{\mathbf X^n_t}{\mathbf{H}}^2\d t<\infty} = 1$ and for each $t \in [0,T]$,
\begin{align*}
[\mathbf X^n_t]_{\PP,\borel{\mathbf{H}}} &= \left[\mathbf{S}_t\boldsymbol \xi + \int_0^t \mathbf{S}_{t-s} \mathbf{F}(\mathbf X^n_s) \d s\right]_{\PP,\borel{\mathbf{H}}} + \int_0^t \mathbf{S}_{t-s} \mathbf{B}(\mathbf X^n_s) P_n \d W_s.
\end{align*}

\item\label{item:wave4} There exists $C \in (0,\infty)$ such that for each $n \in \mathbb N$ and $\phi \in C^2_b(\mathbf{H};\R)\setminus\{0\}$, 
\begin{equation*}
\norm{X_T^\infty- X_T^n}{L^2(\Omega;\mathbf{H})}^2  +
\frac{\abs{ \E{\phi\left(X^\infty_T\right)} -
\E{\phi\left(X^n_T\right)}}}{\norm{\phi}{C^2_b(\mathbf{H};\R)}}
\leq C n^{1-\epsilon}.
\end{equation*}

\end{enumerate}
\end{proposition}

\begin{proof}
\ref{item:wave1}: 
We will prove in two steps that for any $\delta \in \R$ the linear operator 
\begin{equation*}
\mathbf A\colon H_{\delta+1/2} \times H_\delta \subset H_\delta \times H_{\delta-1/2} \to H_\delta \times H_{\delta-1/2}
\end{equation*}
satisfies $\mathbf A^*=-\mathbf A$. First, $\mathbf A^*$ is an extension of $-\mathbf A$ because for each $(v_1,w_1), (v_2,w_2) \in D(\mathbf A)$, 
\begin{align*}
\hspace{2em}&\hspace{-2em}
\langle \mathbf A(v_1,w_1),(v_2,w_2)\rangle_{H_\delta\times H_{\delta-1/2}}
=
\langle w_1,v_2\rangle_{H_\delta} + \langle A v_1,w_2\rangle_{H_{\delta-1/2}}
\\&=
\langle (-A)^{1/2}w_1,(-A)^{1/2}v_2\rangle_{H_{\delta-1/2}} + \langle (-A)^{-1/2}A v_1,(-A)^{-1/2}w_2\rangle_{H_\delta}
\\&=
\langle w_1,-A v_2\rangle_{H_{\delta-1/2}} - \langle v_1,w_2\rangle_{H_\delta}
=
-\langle (v_1,w_1),\mathbf A(v_2,w_2)\rangle_{H_\delta\times H_{\delta-1/2}}.
\end{align*}
Second, to see that $\mathbf A^*=-\mathbf A$ let $(v,w) \in D(\mathbf A^*)$. Then the following linear mapping is bounded:
\begin{equation*}
H_{\delta+1/2}\times H_\delta \subset H_\delta\times H_{\delta-1/2} \to \R, 
\qquad
(h,k) \mapsto \langle \mathbf A(h,k),(v,w)\rangle_{H_\delta\times H_{\delta-1/2}}.
\end{equation*}
Rewriting the last expression as
\begin{align*}
\langle \mathbf A(h,k),(v,w)\rangle_{H_\delta\times H_{\delta-1/2}}
&=
\langle k,v\rangle_{H_\delta} + \langle A h,w\rangle_{H_{\delta-1/2}} 
\\&=
\langle k,v\rangle_{H_\delta} + \langle -(-A)^{1/2} h,(-A)^{-1/2}w\rangle_{H_\delta} 
\end{align*}
and using that $(-\theta \Delta)^{1/2}\colon H_\delta\to H_{\delta-1/2}$ is an isometry shows that the following linear mappings are bounded,
\begin{align*}
H_\delta \subset H_{\delta-1/2} &\to \R, 
&
k &\mapsto \langle k,v\rangle_{H_\delta}, 
\\
H_\delta \subset H_{\delta-1/2} &\to \R, 
&
h &\mapsto \langle h,(-A)^{-1/2}w\rangle_{H_\delta}.
\end{align*}
By \cite[Lemma~3.10.(ii)]{jacobe2015weak} this implies that $v$ and $(-A)^{-1/2}w$ belong to $H_{\delta+1/2}$, which is equivalent to $(v,w) \in H_{\delta+1/2}\times H_\delta=D(\mathbf A)$. This proves that $\mathbf A^*=-\mathbf A$. It follows from a theorem of Stone \cite[Theorem~3.24]{engel1999one} that $\mathbf A$ generates a strongly continuous group of isometries on $H_\delta\times H_{\delta-1/2}$. As this holds true for $\delta=\eta$ and $\delta = \rho$, we have proven \ref{item:wave1}. 
We now show \ref{item:wave2}: For each $x \in H_\eta$ let $F(x)\colon[0,1]\to\R$ and $B(x)\colon[0,1]\to\R$ be the mappings which satisfy for each $s \in (0,1)$ that 
\begin{equation*}
F(x)(s) = f(s,x(s)), \qquad B(x)(s) = b_0(s) + b_1(s) x(s).
\end{equation*}
We claim that $\mathbf F \in C^2_b(\mathbf{H})$ and $\mathbf{F}|_{\mathbf{V}} \in \Lip(\mathbf{V})$. As $\mathbf{F}(x_1,x_2)= (0,F(x_1))$, it is sufficient to show that $F \in C^2_b(H_\eta;H_{\eta-1/2})$ and $F|_{H_\rho} \in \Lip(H_\rho;H_{\rho-1/2})$. This can be seen as follows under assumptions \ref{item:wavea} or \ref{item:waveb}:
\begin{itemize}
\item[\ref{item:wavea}] 
Recall that $\eta=0$. The function $f_0$ belongs to $H_{-1/2}\cap H_{\rho-1/2}=H_{\rho-1/2}$ by definition. Moreover, multiplication $x\mapsto f_1 x$ belongs to $L(H)$ and therefore also to $L(H;H_{-1/2})$ and $L(H_\rho;H_{\rho-1/2})$. This proves the claim in the case \ref{item:wavea}. 
\item[\ref{item:waveb}]
For each $\delta \in (0,1/4)$ we have $H_\delta=H^{2\delta}$ by Lemma~\ref{lem:interpolation}.\ref{item:interpolation1}, $F \in C^2_b(H_{\delta}; L^1((0,1)))$ by Lemma~\ref{lem:nemytskiismooth}, and $F \in C^2_b(H_{\delta}; H_{\delta-1/2})$ by Lemma~\ref{lem:interpolation}.\ref{item:interpolation2}. Choosing $\delta = \rho$ and $\delta=\eta$ proves the claim in the case \ref{item:waveb}.
\end{itemize}

We claim that $\mathbf{B} \in C^2_b(\mathbf{H};\HS(U;\mathbf{H}))$ and $\mathbf{B}|_{\mathbf{V}} \in \Lip(\mathbf{V};\HS(U;\mathbf{V}))$. The conditions on $b_0,b_1$ guarantee that $B \in C^2_b(H_\eta)$ and $B|_{H_\rho} \in \Lip(H_\rho)$ by Lemma~\ref{lem:interpolation}.\ref{item:interpolation1} and Lemma~\ref{lem:multiplication}.\ref{item:multiplication1}. 
As $\mathbf{B}(x_1,x_2)(u) = (0,B(x_1)u)$, it remains to show that the multiplication operator $M\colon x\mapsto (u\mapsto xu)$ belongs to $L(H_\eta;\HS(U;H_{\eta-1/2}))$ and $L(H_\rho;\HS(U;H_{\rho-1/2}))$. 
We will prove the more general statement that $M$ belongs to $L(H_\gamma;\HS(H_{-\gamma};H_\beta))$ for each $\gamma \in [0,1/4)$ and $\beta \in (-\infty,-1/4-\gamma)$. This can be seen as follows. 
By Lemma~\ref{lem:interpolation}.\ref{item:interpolation1} and Lemma~\ref{lem:multiplication}.\ref{item:multiplication2} the following is a finite constant for each $\gamma \in [0,1/4)$, 
\begin{align*}
C_\gamma:=\sup_{\substack{0\neq f\in W^{1,\infty}((0,1))\\0\neq u\in H_\gamma}} \frac{\norm{fu}{H_\gamma}}{\norm{u}{H_\gamma}\norm{f}{L^\infty((0,1))}^{1-2\gamma} \norm{f}{W^{1,\infty}((0,1))}^{2\gamma}}<\infty.
\end{align*} 
Therefore, it holds true for each $x \in H_\gamma$ that 
 \begin{equation*}
   \begin{aligned}
 \hspace{2em}&\hspace{-2em}
 \norm{M(x)}{\HS(H_{-\gamma};H_\beta)}^2
 =
 \sum_{n\in\N} \norm{(-A)^\gamma e_n}{H}^2 \norm{xe_n}{H_\beta}^2
 =
 \sum_{n\in\N} \norm{x (-A)^\gamma e_n}{H_\beta}^2
 =
 \sum_{n\in\N} \|(-A)^\beta (x (-A)^\gamma e_n)\|_{H}^2
 \\&=
 \sum_{n\in\N}\sum_{m \in \N} \abs{\ip{ (-A)^\beta (x (-A)^\gamma e_n)}{e_m}{H}}^2
 =
 \sum_{n\in\N}\sum_{m \in \N} \abs{\ip{ x (-A)^\gamma e_n}{(-A)^\beta e_m}{H}}^2
 \\&=
 \sum_{n\in\N}\sum_{m \in \N} \norm{(-A)^\beta e_m}{H}^2 \abs{\ip{ x (-A)^\gamma e_n}{e_m}{H}}^2
 =
 \sum_{n\in\N}\sum_{m \in \N} \norm{(-A)^\beta e_m}{H}^2 \abs{\ip{(-A)^\gamma e_n}{x e_m}{H}}^2
 \\&=
 \sum_{m \in \N} \norm{(-A)^\beta e_m}{H}^2 \norm{x e_m}{H_\gamma}^2
 \leq
 C_\gamma^2 \sum_{m \in \N} \norm{(-A)^\beta e_m}{H}^2 \norm{x}{H_\gamma}^2 \norm{e_m}{L^\infty((0,1))}^{2(1-2\gamma)} \norm{e_m}{W^{1,\infty}((0,1))}^{4\gamma}
 \\&=
 2 C_\gamma^2 \norm{x}{H_\gamma}^2 \sum_{m \in \N}
 (m\pi)^{4\beta}  (1+m)^{4\gamma} 
 <\infty.
 \end{aligned}
 \end{equation*}
This proves that $M \in L(H_\gamma;\HS(H_{-\gamma};H_\beta))$, and we have shown \ref{item:wave2}. \ref{item:wave3} follows from \ref{item:wave1} and \ref{item:wave2} by Lemma~\ref{lem:apriori}.\ref{item:apriori1}.
We now show \ref{item:wave4}: For each $n \in \N \cup \{\infty\}$ the conditions of Theorem~\ref{thm:weak_errorG} are satisfied thanks to \ref{item:wave1}, \ref{item:wave2}, and \ref{item:wave3}. The convergence rate provided by Theorem~\ref{thm:weak_errorG} can be estimated as follows. As $\eta<\rho-1/4$, we have $M \in L(H_\rho;\HS(H_{-\rho};H_{\eta-1/2}))$, as shown in the proof of \ref{item:wave2}. Therefore,
\begin{align*}
\hspace{2em}&\hspace{-2em}
\sup_{(x_1,x_2) \in \mathbf{V}} \frac{\sum_{k=n}^\infty \norm{\mathbf{B}(x_1,x_2)e_k}{\mathbf{H}}^2 }{1+\norm{(x_1,x_2)}{\mathbf{V}}^2 } 
\leq
n^{\epsilon-1} \sup_{(x_1,x_2) \in \mathbf{V}} \frac{\sum_{k\in\N} k^{1-\epsilon} \norm{\mathbf{B}(x_1,x_2)e_k}{\mathbf H}^2 }{1+\norm{(x_1,x_2)}{\mathbf V}^2 } 
\\&=
n^{\epsilon-1} \sup_{x \in H_\rho} \frac{\sum_{k\in\N} \norm{ e_k}{H_{\rho}}^2 \norm{B(x)e_k}{H_{-1/2}}^2 }{1+\norm{x}{H_\rho}^2 }
=
n^{\epsilon-1} \sup_{x \in H_\rho} \frac{\norm{M(B(x))}{\HS(H_{-\rho};H_{-1/2})}^2}{1+\norm{x}{H_\rho}^2 }
\\&\leq
n^{\epsilon-1} \norm{M}{L(H_\rho;\HS(H_{-\rho};H_{-1/2}))}^2 \norm{B}{\Lip(H_\rho)}^2<\infty.
\qedhere
\end{align*}
\end{proof}

\begin{remark}
Choosing $\mathbf V=\mathbf H$ in Proposition~\ref{prop:wave} leads to a worse convergence rate. 
In fact, this means $\eta =\rho$, and in this case the proof of Proposition~\ref{prop:wave}.\ref{item:wave2} requires $\rho<1/8$, which entails $\epsilon=1-4\rho>1/2$. Thus, one gets only a rate of $\approx 1/2$ with $\mathbf V = \mathbf H$ compared to the rate of $\approx 1$ with $\mathbf V \subsetneq \mathbf H$. 
\end{remark}

\subsection{HJMM-type equations}

HJMM-type equations are used to model the stochastic evolution of interest rates. 
Weak error rates of numerical discretizations of HJMM-type equations were studied in \cite{doersek2010semigroup, doersek2013efficient, krivko2013numerical}; see also the references therein. 
The following proposition provides an upper bound on the weak error of noise discretizations of HJMM equations with additive noise, i.e., of infinite-dimensional Ornstein--Uhlenbeck forward rate models.
The result does not generalize to multiplicative noise because this would lead to a quadratic term in the drift and to explosion of the solution in finite time \cite[Section~6.4.1]{filipovic2001consistency}.  
To ensure that the noise discretization preserves the HJMM-condition and thereby the absence of arbitrage, we discretize the drift together with the volatility. 

\begin{proposition}\label{prop:hjmm}
  Let $H =U$ be a separable $\R$-Hilbert space of real-valued functions on $[0,\infty)$, 
  let $(\Omega,\mathcal F,(\mathcal F_t)_{t\in [0,T]},\mathbb P)$ be a stochastic basis, 
  let $(W_t)_{t\in [0,T]}$ be an $\Id_U$-cylindrical $(\F_t)_{t\in [0,T]}$-Wiener process, 
  let $S\colon[0,\infty)\to L(H)$ be a strongly continuous semigroup 
  which satisfies for each $x \in H$ and $s,t \in [0,\infty)$ that $(S_t x)(s)=x(t+s)$, 
  let $H_0$ be a subspace of $H$ consisting of locally integrable functions, let $m\colon H_0\times H_0 \to H$ be a bounded bilinear mapping which satisfies for all $x,y \in H_0$ and $t \in [0,\infty)$ that
  \begin{equation*}
    m(x,y)(t) = x(t) \int_0^t y(s)\d s,
  \end{equation*}
  let $B \in \HS(U;H_0)$, let $(e_k)_{k\in \N}$ be an orthonormal basis of $U$, 
  for each $n \in \N\cup\{\infty\}$ let $P_n\in L(U)$ denote the orthogonal projection onto the closure of $\operatorname{span}_{\R}\{e_k: k \in \N \cap [0,n)\}$, and let $\xi \in \cL^2(\Omega;H)$ be $\F_0\slash \borel{H}$-measurable. Then the following statements hold:
  \begin{enumerate}[label=(\roman*)]
  \item\label{item:hjmm1}
    For each $n \in \N\cup\{\infty\}$ there is an up to modifications unique predictable process $X^n\colon[0,T]\times \Omega\to H$ which satisfies that $\P{\int_0^T \norm{X^n_t}{H}^2\d t<\infty} = 1$ and for each $t \in [0,T]$ 
    \begin{equation*}
      [X^n_t]_{\PP,\borel{H}} = \left[S_t\xi + \int_0^t S_{t-s} \trace\big(m(BP_n,BP_n)\big) \d s\right]_{\PP,\borel{H}} + \int_0^t S_{t-s} B P_n \d W_s.
    \end{equation*}
    
  \item\label{item:hjmm2}
    There exists $C \in \mathbb R$ such that for each $n \in \mathbb N$ and $\phi \in C^2_b(H;\R)\setminus\{0\}$, 
    \begin{equation*}
      \norm{X_T^\infty- X_T^n}{L^2(\Omega;H)}^2  +
      \frac{\abs{ \E{\phi\left(X^\infty_T\right)} -
          \E{\phi\left(X^n_T\right)}}}{\norm{\phi}{C^2_b(H;\R)}}
      \leq C \sum_{k=n}^\infty \norm{Be_k}{H}^2.
    \end{equation*}
  \end{enumerate}
\end{proposition}

\begin{proof}
\ref{item:hjmm1} follows from Lemma~\ref{lem:apriori}.\ref{item:apriori1}.

\ref{item:hjmm2}: By Lemma~\ref{lem:apriori}.\ref{item:apriori1} there exists for each $n \in \N \cup\{\infty\}$ an up to modifications unique predictable process $Y^n\colon[0,T]\times\Omega\to H$ which satisfies that $\P{\int_0^T \norm{Y^n_t}{H}^2\d t<\infty} = 1$ and for each $t \in [0,T]$,
\begin{align*}
[Y^n_t]_{\PP,\borel{H}} &= \left[S_t\xi + \int_0^t S_{t-s} \trace\big(m(B,B)\big) \d s\right]_{\PP,\borel{H}} + \int_0^t S_{t-s} B P_n \d W_s.
\end{align*}
As $X^\infty=Y^\infty$ one has for each $\phi \in C^2_b(H;\R)\setminus\{0\}$ and $n \in \N$ that
\begin{align*}
\hspace{2em}&\hspace{-2em}
\norm{X_T^\infty- X_T^n}{L^2(\Omega;H)}^2  +
\frac{\abs{ \E{\phi\left(X^\infty_T\right)} -
\E{\phi\left(X^n_T\right)}}}{\norm{\phi}{C^2_b(H;\R)}}
  \\&\leq 
2\norm{Y_T^\infty- Y_T^n}{L^2(\Omega;H)}^2 + \frac{ \abs{ \E{\phi\left(Y^\infty_T\right)} - \E{\phi\left(Y^n_T\right)}} }{\norm{\phi}{C^2_b(H;\R)}} 
\\&\qquad+
2\norm{Y_T^n- X_T^n}{L^2(\Omega;H)}^2
+ \frac{ \abs{ \E{\phi\left(Y^n_T\right)} - \E{\phi\left(X^n_T\right)}} }{\norm{\phi}{C^2_b(H;\R)}}
\\&\leq 
2\norm{Y_T^\infty- Y_T^n}{L^2(\Omega;H)}^2 + \frac{ \abs{ \E{\phi\left(Y^\infty_T\right)} - \E{\phi\left(Y^n_T\right)}} }{\norm{\phi}{C^2_b(H;\R)}} 
\\&\qquad+
2\norm{Y_T^n- X_T^n}{H}^2
+ \norm{Y_T^n- X_T^n}{H},
\end{align*}
where it was used in the last step that $\norm{\phi}{\Lip(H;\R)}\leq\norm{\phi}{C^2_b(H;\R)}$ and that $Y^n_T-X^n_T$ is deterministic. The first two summands on the right-hand side can be bounded using Theorem~\ref{thm:weak_errorG}, and the third and fourth summands can be bounded as follows:
\begin{align*}
\norm{Y^n_T-X^n_T}{H}
&=
\bignorm{\int_0^T S_{T-t}\Big(\trace\big(m(B,B)-\trace\big(m(BP_n,BP_n)\big)\big)\Big) \d t }{H}
\\&\leq
T \norm{S}{B([0,T];L(H))} \norm{\trace\big(m(B(P_\infty+P_n),B(P_\infty-P_n))\big)\big)}{H} 
\\&\leq
2 T \norm{S}{B([0,T];L(H))} \norm{m}{L^{(2)}(H_0;H)}\sum_{k=n}^\infty\norm{Be_k}{H}^2.
\qedhere
\end{align*}
\end{proof}

\subsection{Stochastic Schr\"odinger equations}
Weak error rates for temporal discretizations of Schr\"odinger's equation were established in \cite{debouard2006weak}.   
In the following we provide a convergence rate of the weak error under noise discretizations of Schr\"odinger's equation.  
The equation is formulated on a Hilbert space of complex-valued functions. 
Nevertheless, we do not require the coefficients of the equation to be complex differentiable because this would be overly restrictive, and real differentiability is sufficient for our purpose. 
Thus, we will treat all complex Hilbert spaces, including the field of complex numbers itself, as Hilbert spaces over the real numbers. 

\begin{proposition}\label{prop:schrodinger}
Let $\C_\R=\left\{(\begin{smallmatrix}a&-b\\b&a\end{smallmatrix})\in\mathbb R^{2\times 2}\colon a,b \in \mathbb R\right\}$ be the $\R$-Hilbert space with inner product given by $\ip{(\begin{smallmatrix}a&-b\\b&a\end{smallmatrix})}{(\begin{smallmatrix}c&-d\\d&c\end{smallmatrix})}{\C_\R}=ac+bd$ for all $(\begin{smallmatrix}a&-b\\b&a\end{smallmatrix}), (\begin{smallmatrix}c&-d\\d&c\end{smallmatrix}) \in \C_\R$,
let $i = (\begin{smallmatrix}0&-1\\1&0\end{smallmatrix})\in\C_\R$,
let $d \in \N$, 
let $\alpha \in (0,\infty)$, 
let $\epsilon \in (0,\infty)$,
let $r>d/2$, 
let $H=U=L^2(\R^d;\C_\R)$, 
let $\Delta\colon H^2(\R^d;\C_\R)\subset H\to H$ be the Laplace operator, acting componentwise,  
let $V=H^r(\R^d;\C_\R)$, 
let $f_0, f_1, b_0,b_1 \in V$, 
let $\Sigma \in \HS(U;V)$,
let $(\Omega,\mathcal F,(\mathcal F_t)_{t\in [0,T]},\mathbb P)$ be a stochastic basis, 
let $(W_t)_{t\in [0,T]}$ be an $\Id_U$-cylindrical $(\F_t)_{t\in [0,T]}$-Wiener process, 
let $(e_k)_{k\in\N}$ be an orthonormal basis of $U$, 
and for each $n\in\N\cup\{\infty\}$ let $P_n \in L(U)$ be the orthogonal projection onto the closure of the linear span of $\{e_k: k\in \N\cap [0,n)\}$,
and let $\xi \in \cL^2(\Omega;V)$ be $\F_0\slash\borel{V}$-measurable.
Then the following statements hold true.
\begin{enumerate}[label=(\roman*)]
\item\label{item:schrodinger1} The operator $-i\Delta$ generates a strongly continuous group of isometries $S\colon\R\to L(H)$, which restricts to a strongly continuous group of isometries $S|_V\colon\R\to L(V)$. 

\item\label{item:schrodinger2} There are unique mappings $F\in C^2_b(H;H)$ and $B \in C^2_b(H;\HS(U;H))$ which satisfy for all $u,v \in H$ and $x \in \R^d$ that 
\begin{equation*}
F(v)(x) = f_0(x)+f_1(x)v(x), \qquad B(v)(u)(x) = \big(b_0(x) + b_1(x) v(x)\big)(\Sigma u)(x).
\end{equation*}
The mappings $F$ and $B$ restrict to $F|_V \in \Lip(V;V)$ and $B|_V \in \Lip(V;\HS(U;V))$.

\item\label{item:schrodinger3} For each $n \in \N\cup\{\infty\}$ there is an up to modifications unique predictable process $X^n\colon[0,T]\times \Omega\to H$ which satisfies for each $t \in [0,T]$ that $\P{\int_0^T \norm{X^n_t}{H}^2 \d t<\infty} = 1$ and
\begin{align*}
[X^n_t]_{\PP,\borel{H}} &= \left[S_t\xi -i \int_0^t S_{t-s} F(X^n_s) \d s\right]_{\PP,\borel{H}} - i \int_0^t S_{t-s} B(X^n_s) P_n \d W_s.
\end{align*}

\item\label{item:schrodinger4} There exists $C \in \mathbb R$ such that for each $n \in \mathbb N$ and $\phi \in C^2_b(H;\R)\setminus\{0\}$, 
\begin{equation*}
\norm{X_T^\infty- X_T^n}{L^2(\Omega;H)}^2  +
\frac{\abs{ \E{\phi\left(X^\infty_T\right)} -
\E{\phi\left(X^n_T\right)}}}{\norm{\phi}{C^2_b(H;\R)}}
\leq C \sum_{k=n}^\infty \norm{\Sigma e_k}{H}^2.
\end{equation*}
\end{enumerate}
\end{proposition}

\begin{proof}
  \ref{item:schrodinger1}: As the mapping $\C\ni \lambda \mapsto (\begin{smallmatrix}\Re \lambda &-\Im \lambda\\ \Im \lambda&\Re \lambda\end{smallmatrix}) \in \C_\R$ is an isometric isomorphism of $\R$-Hilbert spaces, the statement is equivalent to the following well-known fact: the linear operator\linebreak
$\sqrt{-1}\Delta \colon H^2(\R^d;\C) \subset L^2(\R^d;\C) \to L^2(\R^d;\C)$
generates the strongly continuous group of isometries $G\colon \R \to L(L^2(\R^d;\C))$ which satisfies for all $t \in \R$, $v \in L^2(\R^d;\C)$ and $\xi \in \R^d$ that
$\widehat{G_t v}(\xi) = \exp(\sqrt{-1} t \xi^2) \widehat{v}(\xi)$,
and $G|_{H^r(\R^d;\C)}\colon\R\to L(H^r(\R^d;\C))$ is a strongly continuous group of isometries. 

\ref{item:schrodinger2}: This follows from the fact that $V$ is a Banach algebra, and multiplication $V\times H\to H$ is bounded bilinear, as shown in Lemma~\ref{lem:multiplication}.\ref{item:multiplication1}. 
\ref{item:schrodinger3} follows from \ref{item:schrodinger1} and \ref{item:schrodinger2} by Lemma~\ref{lem:apriori}.
\ref{item:wave4} follows from Theorem~\ref{thm:weak_errorG} and the estimate
\begin{align*}
\sup_{x \in V} \frac{\sum_{k=n}^\infty \norm{B(x)e_k}{H}^2 }{1+\norm{x}{V}^2 }
&\leq
\sup_{x \in V} \frac{\norm{b_0+b_1 x}{V}^2}{1+\norm{x}{V}^2 } \sum_{k=n}^\infty \norm{\Sigma e_k}{H}^2, 
\end{align*}
noting that the first factor is finite because $V$ is a Banach algebra. 
\end{proof}

\begin{remark}\label{rmk:schroedinger}
Proposition~\ref{prop:schrodinger} remains valid if the condition $f_1,b_1 \in V$ is replaced by $f_1,b_1 \in \R$. In fact, all that is needed is that the mappings $x\mapsto f_1 x$ and $x\mapsto b_1 x$ belong to $L(H)$ and $L(V)$. 
Furthermore, nonlinear Nemytskii operators $F$ could be accommodated by increasing the Sobolev regularity of $H$, which however potentially lowers the convergence rate. 
\end{remark}

\subsection{Linearized stochastic Korteweg--de Vries equation}

Due to its non-linearity the Korteweg--de Vries equation is not
directly amenable to the current methods of numerical weak error
analysis, but its linearization, which is sometimes called Airy's equation, is. 
In the following we establish a weak convergence rate for the
discretization of additive and multiplicative noise.

\begin{proposition}\label{prop:kdv}
Let $\alpha \in (0,\infty)$, 
let $\epsilon \in (0,\infty)$, 
let $r>1/2$, 
let $H=U=L^2(\R)$, 
let $A\colon H^3(\R)\subset H\to H$ be the linear operator which satisfies for all $v \in H^3(\R)$ and $x \in \R$ that $Av(x)=-v'''(x)$, 
let $V=H^r(\R)$, 
let $f_0, f_1, b_0, b_1 \in V$, 
let $\Sigma \in \HS(U;V)$, 
let $(\Omega,\mathcal F,(\mathcal F_t)_{t\in [0,T]},\mathbb P)$ be a stochastic basis, 
let $(W_t)_{t\in [0,T]}$ be an $\Id_U$-cylindrical $(\F_t)_{t\in [0,T]}$-Wiener process, 
let $(e_k)_{k\in\N}$ be an orthonormal basis of $U$, 
for each $n\in\N\cup\{\infty\}$ let $P_n \in L(U)$ be the
orthogonal projection onto the closure of the linear span of
$\{e_k: k\in \N\cap [0,n)\}$, 
and let $\xi \in \cL^2(\Omega;V)$ be $\F_0\slash\borel{V}$-measurable.
Then the following statements hold true.
\begin{enumerate}[label=(\roman*)]
\item\label{item:kdv1} The operator $A$ generates a strongly continuous group of isometries $S\colon \R\to L(H)$, which restricts to a strongly continuous group of isometries $S|_V\colon\R\to L(V)$. 

\item\label{item:kdv2} There are unique mappings $F\in C^2_b(H)$ and $B \in C^2_b(H;\HS(U;H))$ which satisfy for all $u,v \in H$ and all $x \in \R^d$ that 
\begin{equation*}
F(v)(x) = f_0(x)+f_1(x)v(x), \qquad B(v)(u)(x) = \big(b_0(x) + b_1(x) v(x)\big)(\Sigma u)(x).
\end{equation*}
The mappings $F$ and $B$ restrict to $F|_V \in \Lip(V)$ and $B|_V \in \Lip(V;\HS(U;V))$.

\item\label{item:kdv3} For each $n \in \N\cup\{\infty\}$ there is an up to modifications unique predictable process $X^n\colon[0,T]\times \Omega\to H$ which satisfies that $\P{\int_0^T \norm{X^n_t}{H}^2 \d t<\infty} = 1$ and for each $t \in [0,T]$
\begin{equation*}
[X^n_t]_{\PP,\borel{H}} = \left[S_t\xi + \int_0^t S_{t-s} F(X^n_s) \d s\right]_{\PP,\borel{H}} + \int_0^t S_{t-s} B(X^n_s) P_n \d W_s.
\end{equation*}

\item\label{item:kdv4} There exists $C \in \mathbb R$ such that for each $n \in \mathbb N$ and $\phi \in C^2_b(H;\R)\setminus\{0\}$, 
\begin{equation*}
\norm{X_T^\infty- X_T^n}{L^2(\Omega;H)}^2  +
\frac{\abs{ \E{\phi\left(X^\infty_T\right)} -
\E{\phi\left(X^n_T\right)}}}{\norm{\phi}{C^2_b(H;\R)}}
\leq C \sum_{k=n}^\infty \norm{\Sigma e_k}{H}^2.
\end{equation*}
\end{enumerate}
\end{proposition}

\begin{proof}
\ref{item:kdv1}: Let $\hat{}$ denote the Fourier transform, and let $i =\sqrt{-1}$. For each $v \in H$ and $t \in [0,\infty)$ let $S_tv$ be the unique element of $H$ which satisfies for each $\xi \in \R$ that $\widehat{S_t v}(\xi) = \exp(i t \xi^3) \widehat{v}(\xi)$. 
This defines a strongly continuous group of isometries on $H$, whose generator is $A$. Moreover, $S$ restricts to a strongly continuous group of isometries on $V$. 

\ref{item:kdv2} can be seen in the same way as Proposition~\ref{prop:schrodinger}.\ref{item:schrodinger2}.

\ref{item:kdv3} and \ref{item:kdv4} can be shown similarly as in Proposition~\ref{prop:schrodinger}.
\end{proof}

\begin{remark}
Proposition~\ref{prop:kdv} remains valid if the condition $f_1,b_1 \in V$ is replaced by $f_1,b_1 \in \R$. Nonlinear Nemytskii operators $F$ can be treated as mentioned in \autoref{rmk:schroedinger} for the Schr\"odinger equation.
\end{remark}

\appendix

\section{Auxiliary results from interpolation theory}
\label{sec:aux}

We will use several results from interpolation theory to apply the abstract setting of Theorem~\ref{thm:weak_errorG} and Proposition~\ref{prop:weak_error_strong} to the concrete equations in Section~\ref{sec:examples}.  
Throughout this section, we will often write $H^\alpha:= H^\alpha ((0,1))$, $L^p:= L^p((0,1))$, $W^{\alpha,p}:= W^{\alpha,p}((0,1))$, $\alpha\in \R$, $p\in [1,\infty]$.

\subsection{Interpolation spaces of negative order}

Recall that for a (real or complex) Hilbert space $H$ and a symmetric diagonal linear operator $A\colon D(A) \subseteq H \to H$ with $\inf\sigma_P(A)>0$ there exists an up to isometric isomorphisms unique family of interpolation spaces associated to $A$ (see \cite[Theorem~3.5.24]{jentzen2015stochastic} or \cite[Section~3.7]{sell2013dynamics}). 
This is a family of Hilbert spaces $(H_r)_{r\in\R}$ which satisfies for all $v\in H$ and $r, s, t\in \R$ with $r\geq s$ and $t\geq 0$ that $H_r$ is densely contained in $H_s$, $D(A^t) = H_t$, $\norm{\cdot}{H_t} = \norm{A^t(\cdot)}{H}$, and $\norm{v}{H_{-t}} = \norm{A^{-t}v}{H}$.

The following lemma is well-known in the more elaborate setting of sectorial operators (see e.g.~\cite[Theorem~1.18]{lunardi2018interpolation}) and reads as follows in the present simpler setting of diagonal operators.

\begin{lemma}\label{lem:dual}
Let $\mathbb K \in \{\R,\C\}$, 
let $H$ be a $\mathbb K$-Hilbert space, 
let $A\colon D(A)\subseteq H \to H$ be a symmetric diagonal linear operator with $\inf \sigma_P(A)>0$, 
let $(H_r)_{r \in \R}$ be a family of interpolation spaces associated to $A$, 
and let $r \in [0,\infty)$. 
Then there is a unique isometric isomorphism $\phi\colon H_{-r}\to H_r^*$ which satisfies for all $u \in H$ and $v \in H_r$ that $\phi(u)(v) = \ip{u}{v}{H}$.
\end{lemma}

\begin{proof}
Uniqueness of $\phi$ follows from the density of $H$ in $H_{-r}$. It remains to show existence. 
Letting $\hat A^{-r}$ denote the isometric extension of $A^{-r}$ to $H_{-r}$, one has isometries $\hat A^{-r}\colon H_{-r} \to H$ and $A^r\colon H_r \to H$.
Both mappings are surjective: they have closed range because they are isometric, and dense range because their range contains the dense subset $H_r$ of $H$. 
Thus, they are isometric isomorphisms. 
Let $j\colon H\to H^*$ be the Riesz isomorphism, and let $(A^r)^*\colon H^*\to (H_r)^*$ be the Banach space adjoint of $A^r$. 
As $A^r$ is isometric and injective, $(A^r)^*$ is isometric and surjective. 
Then the mapping
$\phi = (A^r)^* \circ j \circ \hat A^{-r}\colon H_{-r}\to (H_r)^*$
is an isometric isomorphism, which satisfies for each $u \in H$ and $v \in H_r$ that
\begin{equation*}
\phi(u)(v) = \ip{\hat A^{-r} u}{A^r v}{H} = \ip{A^{-r} u}{A^r v}{H} = \ip{u}{v}{H}. 
\qedhere
\end{equation*}
\end{proof}

\subsection{Interpolation spaces associated to the Dirichlet Laplacian}

The interpolation spaces of the Dirichlet Laplacian on the unit interval coincide with certain Sobolev spaces. 
This is described in the following lemma, which summarizes several well-known results in this regard.

\begin{lemma}\label{lem:interpolation}
Let $H$ be the $\R$-Hilbert space $L^2((0,1))$, let $\theta\in (0,\infty)$, 
let $\dom(\Delta):= H^2((0,1)) \cap H^1_0((0,1))$ and $\Delta\colon \dom(\Delta) \subset H \to H$ be the
Laplace operator with Dirichlet boundary conditions,
and let $(H_r)_{r \in \R}$ be a family of interpolation spaces associated to $-\theta\Delta$. 
Then the following statements hold true:
\begin{enumerate}[label=(\roman*)]
\item\label{item:interpolation1} 
For each $r \in [0, 3/4) \setminus \{1/4\}$ the spaces $H_r$ and $H_0^{2r}((0,1))$ are equal and carry equivalent norms. 

\item\label{item:interpolation2}
For each $r\in (1/4,\infty)$ the inclusion of $H$ into $H_{-r}$ extends to a unique continuous embedding $\phi\colon L^1((0,1)) \to H_{-r}$.
\end{enumerate}
\end{lemma}

\begin{proof}
The operator $-\theta\Delta$ is symmetric diagonal linear with $\inf\sigma_P(-\theta\Delta)>0$, which implies that the interpolation spaces associated to $-\theta\Delta$ are well-defined. 

\ref{item:interpolation1}: For each $r \in [0,1/4)$ let $H_D^{2r} = H^{2r}$, and for each $r \in (1/4,1)$ let $H_D^{2r}  = \{u \in H^{2r} : u(0)=u(1) = 0\}$. Then the spaces $H_r$ and $H^{2r}_D$ are equal and carry equivalent norms by \cite[Theorem~1.18.10]{triebel1978interpolation} and \cite[Theorem~8.1]{grisvard1967caracterisation}. Moreover, by \cite[Theorem~11.4]{lions1972nonhomogeneous1} it holds for each $r\in [0,1/4)$ that $H_0^{2r} = H^{2r}$, and it holds for each $r \in (1/4,3/4)$ that $H_0^{2r}  = \{u \in H^{2r} \colon u(0)=u(1) = 0\}$. Hence, it holds for each $r\in [0,3/4] \setminus \{1/4\}$ that the spaces $H_r$ and $H^{2r}_0$ are equal and carry equivalent norms. This proves \ref{item:interpolation1}.

\ref{item:interpolation2}: It is sufficient to show the statement for each $r \in (1/4,3/4)$. Uniqueness of $\phi$ follows from the density of $H$ in $L^1$, and existence follows from the following estimate: by Lemma~\ref{lem:dual} one has for each $u \in H$ that 
\begin{equation*}
\norm{u}{H_{-r}} 
= 
\sup_{0\neq v \in H_r} \frac{\abs{\ip{u}{v}{H}}}{\norm{v}{H_r}}
\leq
\norm{u}{L^1}\sup_{0\neq v \in H_r}\frac{\norm{v}{L^\infty}}{\norm{v}{H_r}}, 
\end{equation*}
and the right-hand side is finite by \ref{item:interpolation1} and the Sobolev embedding theorem.
\end{proof}

\subsection{Multiplication operators on Sobolev--Slobodeckij spaces}

The following lemma summarizes some well-known conditions for the continuity of pointwise multiplication of functions with Sobolev--Slobodeckij regularity. 
Point~\ref{item:multiplication2} is specialized to the setting  in~\autoref{prop:wave} where the supremum norms of the multiplier $f$ and its derivative $f'$ can be calculated explicitly. 

\begin{lemma}\label{lem:multiplication}
The following statements hold:
\begin{enumerate}[label=(\roman*)]
\item\label{item:multiplication1}
Let $\alpha,\beta \in [0,\infty)$ and $p,q \in [1,\infty)$ satisfy $\beta\geq \alpha$, $\beta-\alpha\geq 1/q-1/p$, $\beta>1/\min\{p,q\}$. 
Then 
\begin{align*}
\sup_{\substack{0\neq f\in W^{\beta,q}((0,1))\\0\neq g\in W^{\alpha,p}((0,1))}} 
\frac{\norm{fg}{W^{\alpha,p}((0,1))}}{\norm{f}{W^{\beta,p}((0,1))}\norm{g}{W^{\alpha,q}((0,1))}}<\infty.
\end{align*}

\item\label{item:multiplication2}
Let $p\in [1,\infty]$, $\alpha \in [0,1]$. Then
\begin{equation*}
\sup_{\substack{0\neq f\in W^{1,\infty}((0,1))\\0\neq g\in W^{\alpha,p}((0,1))}} \frac{\norm{fg}{W^{\alpha,p}((0,1))}}{\norm{f}{L^\infty((0,1))}^{1-\alpha} \norm{f}{W^{1,\infty}((0,1))}^\alpha\norm{g}{W^{\alpha,p}((0,1))}}<\infty.
\end{equation*}  
\end{enumerate} 
\end{lemma}

\begin{proof}
\ref{item:multiplication1}: 
This is a special case of \cite[Theorem~7.5]{behzadan2015multiplication}.

\ref{item:multiplication2}: 
For each $f \in W^{1,\infty}$ the multiplication operator $M_f\colon g \mapsto fg$ is continuous on the spaces $L^p$ and $W^{1,p}$ and satisfies $\norm{M_f}{L(L^p)} \leq \norm{f}{L^\infty}$ and $\norm{M_f}{L(W^{1,p})} \leq 2\norm{f}{W^{1,\infty}}$ because it holds true for each $g \in L^p$ and $h \in W^{1,p}$ that $\norm{fg}{L^p}\leq \norm{f}{L^\infty} \norm{g}{L^p}$ and
\begin{align*}
  \norm{fh}{W^{1,p}} 
  &\leq \norm{fh}{L^p} +
    \norm{(fh)'}{L^p} 
  \leq 
      \norm{f}{L^\infty} \norm{h}{L^p} + \norm{f'}{L^\infty}\norm{h}{L^p} + \norm{f}{L^\infty} \norm{h'}{L^p} 
  \\&\leq 
      2\norm{f}{W^{1,\infty}}\norm{h}{W^{1,p}}.
\end{align*}
Therefore, $M_f$ acts continuously on the real interpolation space $(L^p, W^{1,p})_{\alpha,p}$, which 
equals $W^{\alpha,p}$ by \cite[Example 1.26]{lunardi2018interpolation} and satisfies
$\norm{M_f}{L(W^{\alpha,p})} \leq 2^\alpha \norm{f}{W^{1,\infty}}^\alpha \norm{f}{L^\infty}^{1-\alpha}$.
\end{proof}

\subsection{Nemytskii operators on Bessel potential spaces}

The following lemma gives sufficient conditions for twice continuous differentiability of certain Nemytskii operators on function spaces below the Sobolev threshold. 
Similar results in slightly different settings can be found in \cite{appell1990nonlinear, runst2011sobolev}.

\begin{lemma}\label{lem:nemytskiismooth}
Let $f \in C^{0,2}_b([0,1]\times\R)$, let $\alpha \in (0,1/2)$, and for all $u \in H^\alpha((0,1))$ and $x \in (0,1)$ let $F(u)(x) := f(x,u(x))$.
Then $F \in C^2_b(H^\alpha((0,1));L^1((0,1)))$.
\end{lemma}

\begin{proof}
Let $p=1/(2\alpha)$ and $q = 2/(1-2\alpha)$. We will use repeatedly that $H^\alpha$ embeds continuously in $L^q$. For each $u,v,w \in H^\alpha$ and $x \in (0,1)$ let 
\begin{align*}
(F'(u)v)(x):=f^{(0,1)}(x,u(x))v(x), 
\qquad
(F''(u)(v,w))(x):=f^{(0,2)}(x,u(x))v(x)w(x).
\end{align*}
We will show below that $F'$ and $F''$ are indeed the derivatives of $F$. For each $u \in H^\alpha$ one has $F(u) \in L^1$ because
\begin{equation*}
\norm{F(u)}{L^1} \leq \mednorm{f^{(0,1)}}{L^\infty((0,1)\times\R)} \norm{u}{L^1} + \norm{f(\cdot,0)}{L^1}<\infty.
\end{equation*}	
For each $u,v \in H^\alpha$ one has $F'(u)v \in L^1$ because
\begin{align*}
\norm{F'(u)v}{L^1} \leq \mednorm{f^{(0,1)}}{L^\infty((0,1)\times\R)} \norm{v}{L^1} <\infty.
\end{align*}
This also shows that $F'\colon H^\alpha\to L(H^\alpha;L^1)$ is bounded. For each $u,v,w \in H^\alpha$ one has $F''(u)(v,w) \in L^1$ because
\begin{align*}
\norm{F''(u)(v,w)}{L^1} \leq \mednorm{f^{(0,2)}}{L^\infty((0,1)\times\R)} \norm{v}{L^2}\norm{w}{L^2} <\infty.
\end{align*}
This also shows that $F''\colon H^\alpha\to
L^{(2)}(H^\alpha;L^1)$ is bounded. Moreover, $F''$ is
continuous. To see this, let $(u_n)_{n\in\N}$ be a sequence which
converges to $u$ in $H^\alpha$. For any sequence
$(n_k)_{k\in\N}$ there exists a subsequence $(n_{k_l})_{l\in\N}$ such
that $u_{n_{k_l}}$ converges to $u$ almost everywhere. By H\"older
inequality, the continuity of $f^{(0,2)}$ and the dominated convergence theorem,
\begin{multline*}
\limsup_{l\to\infty} \mednorm{(F''(u_{n_{k_l}})-F''(u))(v,w)}{L^1} 
\\\leq 
\limsup_{l\to\infty} \mednorm{f^{(0,2)}(\cdot,u_{n_{k_l}}(\cdot))-f^{(0,2)}(\cdot,u(\cdot))}{L^p} \norm{v}{L^q} \norm{w}{L^q}
=0.
\end{multline*}
This implies the continuity of $F''$. The function $F'$ is the
Fr\'echet derivative of $F$. This follows from the following estimate
for $u,v \in H^\alpha$, $v\neq 0$, 
\begin{align*}
\hspace{2em}&\hspace{-2em}
\norm{v}{H^\alpha}^{-1} \norm{F(u+v)-F(u)-F'(u)v}{L^1}
\\&=
\norm{v}{H^\alpha}^{-1} \int_0^1 \abs{\int_0^1 \left(f^{(0,1)}(x,u(x)+tv(x))-f^{(0,1)}(x,u(x))\right) \d t\, v(x)}\d x
\\&\leq
\norm{v}{H^\alpha}^{-1} \mednorm{f^{(0,2)}}{L^\infty((0,1)\times\R)} \norm{v}{L^2}^2,
\end{align*}
noting that for each $u \in H^\alpha$ the right-hand side converges to zero as $v\to 0$ in $H^\alpha$. The function $F''$ is the Fr\'echet derivative of $F'$. This follows from the following estimate for $u,v,w \in H^\alpha$, 
\begin{align*}
\hspace{2em}&\hspace{-2em}
              \norm{F'(u+w)v-F'(u)v-F''(u)(v,w)}{L^1}
  \\&=
      \int_0^1 \abs{\int_0^1 \left(f^{(0,2)}(x,u(x)+tw(x))-f^{(0,2)}(x,u(x))\right) \d t\, v(x) w(x)} \d x
  \\&\leq
      \int_0^1\mednorm{f^{(0,2)}(\cdot,u(\cdot)+tw(\cdot))-f^{(0,2)}(\cdot,u(\cdot))}{L^p} \d t \norm{v}{L^q}\norm{w}{L^q}.
\end{align*}
By the above subsequence argument and dominated convergence theorem
the integral term converges to $0$, as $w\to 0$ in
$H^\alpha$. Using this and that $H^\alpha$ is
continuously embedded in $L^q$ we finally get
\begin{equation*}
\lim_{\substack{w\in H^\alpha\setminus\{0\}\\ w\to 0}}
\frac{\norm{F'(u+w)-F'(u)-F''(u)(\cdot,w)}{L(H^\alpha; L^1)}}{\norm{w}{H^\alpha}}
=  0.  
\qedhere
\end{equation*}
\end{proof}

\printbibliography

\end{document}